\documentclass[11pt]{amsart}

\usepackage{amsmath, amsthm, amssymb, euscript, enumerate}
\usepackage{pxfonts, graphicx, subfigure, epsfig}
\usepackage{mathtools, ulem, appendix, accents, color}

\usepackage{amssymb,mathrsfs,graphicx,enumerate}
\title{Abstracts}

\newcommand{\Poincare}{{Poincar\'e}}

\theoremstyle{plain}
\newtheorem{theorem}{Theorem}[section]
\newtheorem{definition}{Definition}[section]

\newtheorem{lemma}[theorem]{Lemma}
\newtheorem{proposition}[theorem]{Proposition}
\newtheorem{exmp}{Example}[section]

\newtheorem*{remark}{Remark}
\numberwithin{equation}{section}

\title{Some Uniformization Problems for a Fourth order Conformal Curvature}

\author{Sanghoon Lee}
\address{Department of Mathematics, Princeton University, Princeton, NJ 08544, USA }
\email{sl29@math.princeton.edu }

\begin{document}

\date{\today}


\maketitle
\begin{abstract}
In this paper, we establish the existence of conformal deformations that uniformize fourth order curvature on 4-dimensional Riemannian manifolds with positive conformal invariants. Specifically, we prove that any closed, compact Riemannian manifold with positive Yamabe invariant and total $Q$-curvature can be conformally deformed into a metric with positive scalar curvature and constant $Q$-curvature. For a Riemannian manifold with umbilic boundary, positive first Yamabe invariant and total $(Q, T)$-curvature, it is possible to deform it into two types of Riemannian manifolds with totally geodesic boundary and positive scalar curvature. The first type satisfies $Q\equiv \text{constant}, T \equiv 0$ while the second type satisfies $Q\equiv 0, T \equiv \text{constant}$.
\end{abstract}

\section{Introduction}

In conformal geometry, the $Q$-curvature of a 4-dimensional Riemannian manifold $(M, g)$ is a curvature invariant associated to the Paneitz operator $P_4$. They are defined respectively by 

\begin{equation}
\label{eq:qcurvature}
 Q = \frac{1}{12}(-\Delta R + R^2 - 3|Ric|^2)
\end{equation}

\begin{equation}
\label{eq:paneitz}
P_4 \phi= \Delta^2 \phi - \mathrm{div}(\frac{2}{3}R g - 2Ric)d\phi
\end{equation}
for any smooth function $\phi$ on $M$. Here, $R$ is the scalar curvature, $Ric$ is the Ricci curvature, $\Delta$ is the Laplace-Beltrami operator, and $\mathrm{div}$ is minus of the adjoint of the exterior derivative $d$. Let $g_w \coloneqq g^{2w}g$ be the conformally deformed metric where $w$ is a smooth function on $M$.
The Paneitz operator the $Q$-curvature exihibit the following conformal covariance property:
\begin{equation*}
(P_4)_w(\phi) = e^{-4w} P_4(\phi), \, P_4 w + 2Q = 2Q_w e^{4w}
\end{equation*}
where $(P_4)_w$ and $Q_w$ are the Paneitz operator and the $Q$-curvature for the metric $g_w$. They are defined in any dimensions greater than or equal to three as discussed in \cite{HY} and other related references. However, an especially interesting aspect in dimension four is that $Q$-curvature is related to the topology of $M$ through the Chern-Gauss-Bonnet formula:
\begin{equation*}
\int_M (Q + \frac{|W|^2}{8} ) dV = 4\pi^2 \chi(M),
\end{equation*}
where $W$ is the Weyl curvature and $\chi(M)$ is the Euler characteristic of $M$. The total $Q$-curvature defined by $k_p \coloneqq \int_M Q dV$ is a conformal invariant. In this regard, $Q$-curvature serves as a four-dimensional analogue of Gaussian curvature of Riemannian surfaces, as both the Laplace-Beltrami operator and the Gaussian curvature exhibit similar conformal covariance properties. The Chern-Gauss-Bonnet formula involving the $Q$-curvature is considered as a four-dimensional version of the Gauss-Bonnet formula.

The study of the analytic aspects and applications to geometry and topology of the $Q$-curvature and the Paneitz operator has been a topic of extensive research in recent decades. Notable works in this field include Chang and Yang \cite{CY}, Chang-Gursky-Yang \cite{CGY}, Djadli and Malchiodi \cite{DM}, Gursky \cite{G1}, \cite{G2}, and Malchiodi \cite{M}.

In the context of manifolds with boundary, represented as $(M,\partial M, g)$, a natural boundary operator $P_3$ and its associated curvature invariant, the $T$-curvature, have been introduced by Chang and Qing \cite{CQ1} and \cite{CQ2}. They are defined as follows:
\begin{equation}
\label{eq:p3operator}
P_3 \phi = \frac{1}{2} \frac{\partial \Delta \phi}{\partial n} + \Delta_{\hat{g}} \frac{\partial  \phi}{\partial n} - 2H \hat{\Delta}  \phi + L(\hat{\nabla}  \phi, \hat{\nabla}  \phi ) + \hat{\nabla} H \cdot \hat{\nabla}  \phi + ( F - \frac{R}{3} ) \frac{\partial  \phi}{\partial n}
\end{equation}

\begin{equation}
\label{eq:tcurvature}
T = -\frac{1}{12} \frac{\partial R}{\partial n} + \frac{1}{2} R\cdot H - \langle G, L\rangle + 3H^3 - \frac{1}{3} Tr(L^3) + \hat{\Delta} H 
\end{equation}
for any smooth function $\phi$ on $M$. Here, $\hat{\Delta}$, $\hat{\nabla}$ are the Laplace-Beltrami operator and covariant derivative with respect to the induced metric on the boundary $\hat{g} = g|_{\partial M}$. $L$ is the second fundamental form of the boundary, $H$ is the mean curvature, $F = R_{anan}$, $\langle G, L \rangle = R_{anbn}L_{ab}$ where $n$ is the inward normal vector field along the boundary and $a, b$ denote tangential components along the boundary.

They also have the conformal covariance property,
\begin{equation}
(P_3)_w = e^{-3w} P_3, P_3w + T = T_w e^{3w}.
\end{equation}

In four dimension, we the have the following Chern-Gauss-Bonnet formula involving $Q$-curvature and $T$-curvature:
\begin{equation}
\int_M (Q + \frac{|W|^2}{8} ) dV_g + \oint_{\partial M} (T-\mathcal{L}_4 - \mathcal{L}_5) dS_g = 4\pi^2 \chi(M)
\end{equation}
where $\mathcal{L}_4$ and $\mathcal{L}_5$ are curvature quantities defined explicitly in \cite{CQ1}.
As we will be focusing on manifolds with totally geodesic boundaries in this paper, the pointwise conformally invariant quantity $(\mathcal{L}_4 + \mathcal{L}_5)dS_g$  is not relevant to our discussion as it vanishes in such cases.
Also note that the total $(Q, T)$-curvature $k_{(P_4, P_3)} \coloneqq \int_M Q dV+ \oint T dS$ is a conformal invariant. 

It is natural to consider the uniformization theorem for $Q$-curvature, and for manifolds with a boundary, the pair $(Q, T)$. This was first studied for closed manifolds in \cite{CY} and subsequently in \cite{DM} for more general cases. Uniformization theorems for $(Q, T)$-curvature have been investigated in works by Catino and Ndiaye \cite{CN}, Ndiaye \cite{N1}, and \cite{N2}.

\bigskip

In this paper, we present a stronger uniformization theorem that not only controls the $Q$-curvature, but also the sign of the scalar curvature under the condition that  the Yamabe constant and total $Q$-curvature are positive.The Yamabe constant, defined for a 4-dimensional Riemannian manifold $(M, g)$ as 
\begin{equation}
Y(M, g) = \inf_{\tilde{g} \in [g] } \frac{\int_M \tilde{R} dVol_{\tilde{g}}}{Vol(M, \tilde{g})^{\frac{1}{2}}}
\end{equation}
is known to be positive if and only if there exists a metric in the conformal class $[g]$ with positive scalar curvature. In dimensions greater than or equal to 6, it has been established that the existence of a conformal metric with positive scalar and $Q$-curvature is equivalent to the positivity of both the Yamabe invariant and the Paneitz operator, as shown in Gursky-Han-Lin \cite{GHL}. In dimension 4, Gursky \cite{G2} proved that the Paneitz operator is non-negative when both the Yamabe invariant and total $Q$-curvature are non-negative. With this in mind, our uniformization theorem can be stated as follows.
\begin{theorem}
\label{thm:main1}
Let $(M, g)$ be a four-dimensional closed manifold with conformal invariants $k_p$ and $Y(M, [g])$ positive. Then there exists a conformal deformation $w$ such that $Q_w$ is constant, and $R_w>0$.
\end{theorem}

For manifolds with a boundary, the first Yamabe invariant is defined as follows

\begin{equation}
Y(M, \partial M ,[g] ) = \inf_{\tilde{g} \in [g] } \frac{\int_M \tilde{R} dVol_{\tilde{g}} + \oint_{\partial M} H dS_{\tilde{g}}}{Vol(M, \tilde{g})^{\frac{1}{2}}}.	
\end{equation}

We present a uniformization theorem for manifolds with an umbilic boundary. Note that an umbilic boundary can always be conformally deformed to a totally geodesic boundary.
\begin{theorem}
\label{thm:main2}
Let $(M, g)$ be a compact manifold with an umbilic boundary and positive conformal invariants $k_{(P_4, P_3)}$ and $Y(M, \partial M, [g])$. Then there exist conformal deformations $w_1$ and $w_2$ such that

$\begin{cases} 
Q_{w_1} \equiv \frac{k_{(P_4, P_3)}}{Vol(M, g_{w_1})}, \, R_{w_1}>0  \text{ on $M$} \\ 
T_{w_1} \equiv 0, \, H_{w_1} \equiv 0  \text{ on $\partial M$}
\end{cases}$
 and    \,
$\begin{cases} 
Q_{w_2} \equiv 0, \, R_{w_2}>0 \text{ on $M$} \\ 
T_{w_2} \equiv \frac{k_{(P_4, P_3)}}{Vol(\partial M, g_{w_2})}, \, H_{w_2} \equiv 0 \text{ on $\partial M$}.
\end{cases}$

\end{theorem}
\bigskip
Now we describe our strategy and plan of this paper to prove the main Theorem 1.1 and Theorem 1.2. In Part 1, we study compact closed manifolds and prove Theorem 1.1. There are two major steps toward the proof of Theorem 1.1. The first step is to deform the background metric using the continuity method so that both the scalar curvature and the $Q$-curvature are positive pointwisely. Then, we apply the Leray-Schauder degree theory to a 1-parameter family of fourth-order equations to establish the existence of the conformal metric described in Theorem 1.1.

In Section 2, we review functionals that were defined in \cite{BO} and appeared while computing explicit formulas for functional determinants of conformally covariant operators in 4-dimensional Riemannian manifolds. These functionals are studied in \cite{CY}, \cite{CGY} and \cite{G2} to derive interesting geometric results. Followng the approach used in \cite{CGY} and  \cite{GHL}, we use the method of continuity for critical points of suitable linear combinations of those functionals.  We prove that we can find a minimizer with positive scalar curvature at the starting functional. The closedness part of the continuity process follows from the uniform $W^{2,2}(M)$-estimate of critical points.

In Section 3, we prove the openness part of the continuity process by showing contant functions are the only solutions  linearized problems. Combining this with the uniform $W^{2,2}(M)$-estimate of critical points, we conclude that the background metric can be conformally deformed into a metric with positive scalar curvature and $Q$-curvature using the method of continuity. The main key ingredient is that the positivity of scalar curvature can be preserved throughout this process. This concludes the first step towards proving the main theorem.

In Section 4, we study a 1-parameter family of equations of the form
\begin{equation}
\label{eq:ls}
P_4 w_t = 2tk_p e^{4w_t} - 2tQ
\end{equation}
subject to the constraint $R_t >0$ and normalization $\int_M e^{4w_t} = 1$. Note that when $t=0$, the equation is simply a linear equation, and that the solution of the equation at $t=1$ gives us the desired conformal deformation of Theorem 1.1. The compactness result is essential to apply the Leray-Schauder degree theory to this 1-parameter family of equations. A more general compactness theorem than we need has already been proved in \cite{M}. The theorem states:
\begin{theorem}(\cite[Theorem 1.1]{M}) \label{theorem:malchiodi}
Suppose  $Ker \, P_4 = \{ \text{constants} \}$ and that $\{ u_l\}_l$ is a sequence of solutions of 
$$P_4 u_l + 2Q_l = 2k_l e^{4u_l}, \,\,\int_M e^{4u_l} dV = 1$$
 with $\{Q_l \}_l$ satisfying $Q_l \rightarrow Q_0$ in $C^0 (M)$. Assume also that $k_p \ne 8k\pi^2$ for $k= 1, 2, \cdots$. Then $\{u_l \}_l$ is bounded in $C^{\alpha} (M)$ for any $\alpha \in (0,1)$.
\end{theorem}

Under our assumptions, we provide an alternative proof by utilizing the concept of a normal metric introduced in \cite{CQY}. 
\begin{definition}(\cite[Definition 3.1]{CQY})
\label{def:normal}
Let $(\mathbb{R}^4, g_0)$ be the standard Euclidean space. A conformal metric $e^{2w}g_0$ satisfying $\int_{\mathbb{R}^4} |Q_w| e^{4w} < \infty$ is defined to be normal if
\begin{equation}
w(x) = \frac{1}{4\pi^2} \int_{\mathbb{R}^4} \log{\frac{|y|}{|x-y|}} Q_w (y) e^{4w(y)} dy + C.
\end{equation}
\end{definition}
We will employ the following criterion to determine whether a metric is normal.
\begin{theorem}(\cite[Theorem 1.4]{CQY}) \label{theorem:normal}
Suppose that the $Q$-curvature of a metric $e^{2w}|dx|^2$ on $\mathbb{R}^4$ is absolutely integrable and suppose that its scalar curvature is nonnegative at infinity. Then the metric is normal.
\end{theorem}

A blow-up analysis shows that the limit conformal factor $w_\infty$ on $\mathbb{R}^4$ satiesfies $R_{w_\infty} \ge 0$. By Theorem \ref{theorem:normal}, we know that $w_\infty$ is normal. The proof for the compactness of solutions is concluded by the following blow-up profile classification of Xu \cite{X}. 

\begin{theorem}(\cite[Theorem 1.2]{X}) \label{theorem:classification}
Suppose $w(x) \in C^1(\mathbb{R}^m)$ is a solution of the equation
\begin{equation}\label{eq:class}
w(x) = \alpha_m \int_{\mathbb{R}^m} \big[\log{\big( \frac{|y|}{|x-y|} \big) } \big]e^{mw(y)}dy + C_0
\end{equation}
for a dimensional constant $\alpha_m$ such that $e^{mw(x)}$ is integrable over $\mathbb{R}^m$. Then there exist a positive constant $\lambda$ and a point $x_0 \in \mathbb{R}^m$ such that 
$$w(x) = \log{\big[\frac{2\lambda}{\lambda^2 + |x-x_0|^2} \big]}.$$
\end{theorem}
Note that the dimensional constant $\alpha_m$ is explicitly computed by substituting the specific solution defined above into the equation (\ref{eq:class}). It is easy to see that $\alpha_m$ is independent of $\lambda$ and $x_0$.

In Section 5, we provide a proof of Theorem 1.1 by using the Leray-Schauder degree theory. We also discuss an example that illustrates a limitation of the method of continuity for the 1-parameter family of equations studied in Section 4.

\bigskip

In Part 2, we investigate compact manifolds with totally geodesic boundary. Since the method is similar to that of closed manifolds, we will briefly outline each section's contents.

The first step is to prove that we can conformally deform the metric so that the curvatures satisfy certain positivity conditions. In Section 6, we review functionals that appeared in \cite{CQ1} and \cite{CQ2} when computing explicit formulas for functional determinants of conformally covariant operators in 4-dimensional Riemannian manifoldswith boundary. Suitable linear combinations of these functionals are used for the method of continuity, similar to Section 2. The existence of minimizers for the starting functional and the closedness of critical points are established in Section 6. In Section 7, we prove the openness part of the continuity method and conclude the proof of the first step.

 In Section 8, we establish the compactness result for 1-parameter families of equations which are Neumann boundary valued versions of (1.9). These compactness results are essential to apply the Lerau-Schauder degree theory. Analogous to the closed manifold case, we will also show that the metrics in these boundary value problems are "normal" in the sense of Definition \ref{def:normal} and use the classification result of Theorem \ref{theorem:classification}. Finally, in Section 9, we provide a proof of Theorem 1.2.

\section*{Acknowledgement}
The author acknowledges the invaluable guidance and support of his advisor Professor Sun-Yung A. Chang throughout the research process. The author also expresses his gratitude to Professor Paul C. Yang for providing helpful insights and engaging in enlightening discussions.

\part{Closed manifolds}

\section{Preliminaries}

Throughout Section 2 and 3, we study critical points of linear combinations of some functionals intoduced in \cite{CY}, \cite{CGY}.

 We assume $k_p = \int_M Q_0 dV_0>0$ and $Y(M, [g_0])>0$. By \cite[Theorem B]{G2}, we have $k_p \le 8\pi^2$, and if $k_p = 8\pi^2$, then $M$ is conformally equivalent to the standard sphere $S^4$. Hence, we may assume that $k_p < 8\pi^2$. We fix the background metric to be a Yamabe metric so that $R_0$ is constant.

Let $\eta$ be a fixed (2,0) tensor defined on $M$ with $|\eta|_0^2>0$ everywhere. We define some functionals as follows:

\begin{align*}
I[w] & = \int 4|\eta|_0^2 wdV_0 - (\int |\eta|_0^2 dV_0) \log{\fint e^{4w} dV_0}\\
II[w] & = \int w (P_4)_0 w dV_0 + \int 4Q_0 wdV_0 -k_p\log{\fint e^{4w} dV_0} \\
& = \int \big[(\Delta_0 w)^2 + \frac{2}{3}R_0|\nabla_0 w|^2 - 2Ric_0(\nabla_0 w, \nabla_0 w) + 4Q_0 w \big] dV_0-k_p\log{\fint e^{4w}dV_0 } \\
III[w] & = \frac{1}{36}(\int R_w^2 dV_w -\int R_0^2 dV_0)  = \int \big[(\Delta_0w + |\nabla_0 w|^2)^2 - \frac{1}{3}R_0 | \nabla_0 w|^2 \big]dV_0
\end{align*}

We consider a one-parameter of family of functionals 
\begin{equation}
\label{eq:functional}
F_t[w] \coloneqq\gamma(t) I[w] + t II[w] + (1-t) III[w]
\end{equation}
where $\gamma(t) = \frac{-tk_p}{\int |\eta|_0^2 dv_0}$. The Euler-Lagrange equation for a critical point $w_t$ and the corresponding metric $g_t \coloneqq e^{2w_t} g_0$ of this functional is as follows:

\begin{equation} \label{eq:el}\tag*{($\ast$)$_{t}$}
\gamma(t) |\eta|_t^2 + tQ_t + \frac{(1-t)}{12}(-\Delta_t R_t ) = 0
\end{equation}
and the weak formulation is as follows; see \cite[Lemma 3.2]{CGY}.
\begin{multline} \label{eq:weak}
\int \Delta_0 w_t \Delta_0 \phi + (1-t)\big[2(\nabla_0 w_t \cdot \nabla_0 \phi)(\Delta_0 w_t + |\nabla_0 w_t|^2 ) + |\nabla_0 w_t|^2 \Delta_0 \phi \big]\\
= \int \big[(\frac{1}{3}-t) R_0g_0 + 2t Ric_0\big](\nabla_0 w_t, \nabla_0 \phi) - \int (2tQ_0 +2 \gamma(t) |\eta|^2_0)\phi
\end{multline}
for $\phi \in W^{2,2}(M)$.

We will vary the parameter $t$ from $\frac{2}{3}$ to $1$. The reason for choosing $\frac{2}{3}$ as the starting point is motivated by the following lemma.

\begin{lemma}\cite[Lemma 1.2]{G1} \label{lemma:g1}
Let $(M, g)$ be a four-dimensional compact Riemannian manifold. Suppose that the scalar curvature $R$ of $g$ satisfies
\begin{equation}
\Delta R - \frac{1}{6} R^2 \le 0.
\end{equation}
Then

(1) if $Y(M, [g])>0$, $R>0$ on $M$.

(2) if $Y(M, [g]) = 0 $, $R \equiv 0$ on $M$.

\end{lemma}
By using the above lemma we can prove that for any critical point of $F_{\frac{2}{3}}$, the deformed metric has positive scalar curvature.
Let $t_0 =\frac{2}{3}$.

\begin{lemma}\label{lemma:starting}
There exsits a smooth minimizer $w_{t_0}$ of $F_{t_0}$. For $w_{t_0}$,  $R_{t_0}$ is positive.
\end{lemma}

\begin{proof}
By \cite[Theorem A]{G2}, it is known that $(P_4)_0$ is non-negative operator on $W^{2,2}(M)$ which also implies that $\int_M w(P_4)_0 w \ge C ||w-\fint_M w||_{W^{2,2}(M)}$ for some constant $C$.
Additionally, $\fint R_w^2dV_w \ge (\fint R_w dV_w)^2 \ge (\fint R_0 dV_0)^2 = \fint R_0^2 dV_0$, where we used the fact that the Yamabe metric minimizes the total scalar curvature when the volume is fixed. Note that there is no $\int_M e^{4w} dV_0$ term in the functional $F_{t_0}$ due to our choice of $\gamma(t_0)$. Therefore $F_{t_0}$ is bounded below and coercive, and we can find a minimizer $w_{t_0}$ with the normalization $\int w_{t_0} dV_0 = 0$. $w_{t_0}$ is smooth by \cite{UV}.

From the equation \ref{eq:el}, we have 
\begin{equation*}
-\Delta_{t_0} R_{t_0} + \frac{1}{6} R_{t_0}^2 = -12\gamma(\frac{2}{3}) |\eta|_{t_0}^2 + 2|E_{t_0}|^2 >0. 
\end{equation*}
where $E_{t_0}$ is the traceless part of the $Ric_{t_0}$.
By Lemma \ref{lemma:g1}, $R_{t_0}$ is positive.

\end{proof}

Next, we observe that any critical point $w_t$ of $F_t$  satisfies $||w_t||_{W^{2,2}} < C$ for some constant independent of $t$ for $\frac{2}{3} \le t \le 1$. This fact will be used to prove closedness part when we apply the method of continuity in Section 3.
\begin{proposition}\label{lemma:weak}
Let $w_t$ be a critical point of the functional $F_t$ with normalization $\overline{w}_t \coloneqq \fint w_t dV_0 =0$  for  $\frac{2}{3} \le t \le 1$. Assume $R_t>0$. Then, $||w_t||_{W^{2,2}} <C$ where $C$ is a constant independent of $t$.
\end{proposition}

\begin{proof}
Substituting $\phi = w_t$ in (\ref{eq:weak}), we have
\begin{multline*}
\int  (\Delta_0 w_t)^2 + (1-t) [3 |\nabla_0 w_t|^2 \Delta_0 w_t + 2 |\nabla_0  w_t|^4] \\
= -(t-\frac{1}{3})\int R_0 |\nabla_0 w_t|^2 + 2t \int Ric_0 (\nabla_0 w_t , \nabla_0 w_t ) - \int (2tQ_0 + 2\gamma(t) |\eta|^2_0 ) w_t
\end{multline*}
Recall the scalar curvature equation
\begin{equation*}
\Delta_0 w_t + |\nabla_0 w_t|^2 + \frac{1}{6} R_t e^{2w_t} = \frac{1}{6} R_0.
\end{equation*}
As $R_t >0$, we have $\int_M |\nabla_0 w_t|^2 = \int_M \Delta_0 w_t + | \nabla_0 w_t|^2 < \frac{1}{6} \int_M R_0 $ which is an uniform bound on $\int |\nabla_0 w_t| ^2 $. 

By Young's inequality, we have $a^2 + 3(1-t)ab + 2(1-t)b^2 \ge [1- \frac{9(1-t)}{8}]a^2 \ge \frac{3}{8} a^2$ for any real number $a, b$ and $t \ge \frac{2}{3}$. Applying this inequality to the left hand side, 
\begin{equation*}
\int  (\Delta_0 w_t)^2 + (1-t) [3 |\nabla_0 w_t|^2 \Delta_0 w_t + 2 |\nabla_0  w_t|^4] \ge \frac{3}{8} \int_M (\Delta_0 w_t)^2
\end{equation*}
For the right hand side,
\begin{align*}
& -(t-\frac{1}{3})\int R_0 |\nabla_0 w_t|^2 + 2t \int Ric_0 (\nabla_0 w_t , \nabla_0 w_t ) - \int (2tQ_0 + 2\gamma(t) |\eta|^2_0 ) w_t \\
&\le 2t |Ric_0|_{L^\infty} \int |\nabla_0 w_t |^2 - \int (2tQ_0 + 2\gamma(t) |\eta|^2_0 ) (w_t - \overline{w}_t) \\
&\le C \int |\nabla_0 w_t |^2 + \frac{1}{2} \int (2tQ_0 + 2\gamma(t) |\eta|^2_0 )^2  + \frac{1}{2}\int  (w_t - \overline{w}_t)^2 \\
&\le C  \int |\nabla_0 w_t |^2 + C
\end{align*}
where at the last line, we used the \Poincare 's inequality. As $\overline{w}_t = \fint w_t = 0$, this gives us an upper bound on the full $W^{2,2}(M)$-norm of $w_t$.
\end{proof}

\section{The method of continuity}

In this section we establish positivity of linearized operators for the equations considered in Section 2 and prove that we can conformally deform the background metric so that both the scalar curvature an $Q$-curvature are positive pointwisely.

Let $w_t$ be a critical point of the functional $F_t$ defined in equation (\ref{eq:functional}), and let $g_t = e^{2w_t}g_0$ be the corresponding metric. Denote the linearization of the equation (\ref{eq:el}) at $g_t$  by $L_t$. When $L_t$ is restricted to the H\"{o}lder space $C^{4,\alpha}(M)$ for $0<\alpha<1$, $L_t:C^{4,\alpha}(M) \rightarrow C^{\alpha}(M)$ is a bounded linear operator. The explicit formula for $L_t$ is derived in the proof of \cite[Theorem 2.1]{CY} and is as follows.

\begin{equation} \label{eq:linearization}
\langle L_t \phi, \phi \rangle = t\langle (P_4)_t \phi, \phi \rangle + (1-t) \big[\int (\Delta_t \phi)^2 - \frac{1}{3} \int R_t |\nabla_t \phi |^2 \big]
\end{equation}
for $\phi \in W^{2,2}(M)$. To prove the opennes part of the continuity method, we need to prove that $\ker L_t$ is trivial when the scalar curvature $R_t$ is positive.

\begin{lemma} \label{lemma:linearization}
Let $w_t$ be a critical point of the functional (\ref{eq:functional}) with positive scalar curvature $R_t$ where $\frac{2}{3} \le t \le 1$. Then, $L_t$ is a non-negative operator on $W^{2,2}(M)$ and $\ker L_t \simeq \mathbb{R}$.
\end{lemma}

\begin{remark}
The proof uses the same techiniques appeared in \cite[Lemma 3.1]{G2} and \cite[Lemma 4.2]{CGY}. The only difference is that the ratio between the coefficients of $II[w]$ and $III[w]$ in the functional being considered lies in a different interval.

\end{remark}

\begin{proof}

Denote the traceless part of the Ricci cuvature of the metric $g_t$ by $E_t$. Let $\phi \in W^{2,2}(M)$. The equation \ref{eq:el} is rewirtten as
\begin{equation}\label{eq:tracelessel}
-\Delta_t R_t +\frac{t}{4} R_t^2 - 3t |E_t|^2 = -12\gamma(t) |\eta|^2_t > 0.
\end{equation}
We have the following inequality using the equation (\ref{eq:tracelessel})
\begin{align*}
-4\int E_t(\nabla_t \phi, \nabla_t \phi) &\ge \int -2\sqrt{3} |E_t| |\nabla_t \phi|^2\\
&\ge \int -2\epsilon \big(\frac{\sqrt{3}}{2}\big)^2 \frac{|E_t|^2}{R_t} |\nabla_t \phi |^2 - 2 \epsilon^{-1} R_t |\nabla_t \phi|^2\\
&\ge \frac{\epsilon}{2t} \int \frac{\Delta_t R_t}{R_t} |\nabla_t \phi_t |^2 - \big( \frac{\epsilon}{8} + \frac{2}{\epsilon}\big) \int R_t |\nabla_t \phi|^2.
\end{align*}
where $\epsilon$ is a positive number to be determined later.

In addition, we have
\begin{align*}
\int \frac{\Delta_t R_t }{R_t} |\nabla_t \phi|^2 & = \int - \nabla_t R_t \cdot \nabla_t(R_t^{-1}) |\nabla_t \phi|^2 - \frac{\nabla_t R_t}{R_t} \cdot \nabla_t |\nabla_t \phi|^2 \\
& = \int \frac{|\nabla_t R_t|^2}{R_t^2} |\nabla_t \phi|^2 - 2 \nabla_t^2 \phi \big(\nabla_t \phi, \frac{\nabla_t R_t}{R_t}\big) \\
& \ge \int \frac{|\nabla_t R_t|^2}{R_t^2} |\nabla_t \phi|^2 - \int \big[|\nabla_t^2 \phi|^2 + \frac{|\nabla_t R_t|^2}{R_t^2} |\nabla_t \phi|^2  \big] \\
& = -\int |\nabla_t^2 \phi|^2 \\
& = \int - (\Delta_t \phi)^2 + E_t(\nabla_t \phi, \nabla_t \phi) + \frac{1}{4} R_t |\nabla_t \phi|^2.
\end{align*}
where at the last line, we used the Bochner formula.

Combining above inequalities, we get 
\begin{equation*}
-4\int E_t(\nabla_t \phi, \nabla_t \phi) \ge   - \frac{\epsilon}{2t} \int (\Delta_t \phi_t)^2 + \frac{\epsilon}{2t} \int E_t(\nabla_t \phi,\nabla_t \phi) - \big(\frac{\epsilon}{8} + \frac{2}{\epsilon} -\frac{\epsilon}{8t}\big) \int R_t |\nabla_t \phi | ^2.
\end{equation*}

Hence we have
\begin{equation} \label{eq:traceless}
(-4-\frac{\epsilon}{2t}) \int E_t(\nabla_t \phi,\nabla_t \phi) \ge -\frac{\epsilon}{2t} \int (\Delta_t \phi)^2 + \big(\frac{\epsilon}{8t} - \frac{\epsilon}{8} - \frac{2}{\epsilon}\big) \int R_t |\nabla_t \phi|^2.
\end{equation}

By plugging the inequality (\ref{eq:traceless}) into the right hand side of the equation (\ref{eq:linearization}), we have

\begin{align*}
\langle L_t \phi, \phi \rangle & = \int (\Delta_t \phi)^2 -2t \int E_t(\nabla_t \phi,\nabla_t \phi) + ( \frac{1}{2} t - \frac{1}{3})\int |\nabla_t \phi|^2 \\
& \ge C_1(\epsilon) \int (\Delta_t \phi)^2  + C_2(\epsilon) \int R_t|\nabla_t \phi|^2
\end{align*}
where $C_1(\epsilon) =  (1 - \frac{\epsilon}{4+\frac{\epsilon}{2t}})$ and $C_2(\epsilon) = \frac{1}{2} t - \frac{1}{3} - (\frac{2t}{4+\frac{\epsilon}{2t}})(\frac{\epsilon}{8} + \frac{2}{\epsilon} -\frac{\epsilon}{8t})$.

Now we check that we can choose $\epsilon$ so that both of coefficients $C_1(\epsilon)$ and $C_2(\epsilon)$ is positive when $\frac{2}{3} \le t \le 1$.

For $C_1(\epsilon)$, we see $1 - \frac{\epsilon}{4+\frac{\epsilon}{2t}} > 0 \iff \frac{4}{1-\frac{1}{2t}} > \epsilon$ and for $C_2(\epsilon)$ we compute: 
 $\frac{1}{2} t - \frac{1}{3} - (\frac{2t}{4+\frac{\epsilon}{2t}})(\frac{\epsilon}{8} + \frac{2}{\epsilon} -\frac{\epsilon}{8t}) > 0 \iff (\frac{1}{2}t - \frac{1}{3})(4+\frac{\epsilon}{2t})\epsilon-2t(\frac{\epsilon}{8} + \frac{2}{\epsilon} -\frac{\epsilon}{8t})\epsilon > 0 \iff (\frac{1}{2} -\frac{t}{4}-\frac{1}{6t})\epsilon^2 + (2t-\frac{4}{3})\epsilon - 4t > 0$. 
The later inequality is quadratic w.r.t $\epsilon$ so it is enough to show the inequality holds for $\epsilon = \frac{4}{1-\frac{1}{2t}}$.

It suffices to check $ 16(\frac{1}{2} -\frac{t}{4}-\frac{1}{6t}) +  4(2t-\frac{4}{3})(1-\frac{1}{2t}) - 4t(1-\frac{1}{2t})^2 > 0$ which is equivalent to $t>\frac{3}{8}$. Thus, we can choose $\epsilon$ so that $C_1(\epsilon)$ and $C_2(\epsilon)$ are both positive. It is trivial to see that $L_t \phi = 0$ if and only if $\phi$ is constant.
\end{proof}

\bigskip

Now we are ready to prove the main proposition of this section which implies that we can conformally deform the background metric so that both the scalar curvature and the $Q$-curvature are positive. The strategy is to apply the continuity method to equation \ref{eq:el}. Roughly speaking, the opennes part follows from the Lemma \ref{lemma:linearization} and the closedness part follows from the Proposition \ref{lemma:weak}.
\begin{proposition} \label{lemma:QR}
Let $(M, g_0)$ be a closed manifold with conformal invariants $k_p$ and $Y(M, [g_0])$ positive. Then, there exists a conformal deformation $w$ such that $Q_w>0$ and $R_w>0$.
\end{proposition}

\begin{proof}
We use the method of continuity to solve the one-parameter family of equations \ref{eq:el}. We define
\begin{equation*}
S = \{t \in [2/3, 1] | \text{ \ref{eq:el} has a smooth solution with positive scalar curvature}\}.
\end{equation*}
For $t = \frac{2}{3}$, we find a minimizer $w_{\frac{2}{3}}$ for the functional $\gamma_1(2/3) I + 8 II + \frac{1}{3} III$. Then by Lemma \ref{lemma:starting}, $g_{\frac{2}{3}}$ has positive scalar curvature. Hence $\frac{2}{3} \in S$. We will show that $S$ is both open and closed to conclude $S = [2/3, 1]$.

First we prove that $S$ is open. Assume $t_1 \in S$. Since the linearization of the equation \ref{eq:el} is positive by Lemma \ref{lemma:linearization}, there is a unique smooth solution $w_t$ of \ref{eq:el} for all $t$ sufficiently close to $t_1$ if we normalize by $\int w_t dV_0 = 0$ by the perturbation theorem. By taking a suffciently small $C^{4, \alpha}(M)$-neighborgood of $w_{t_1}$, we can guarantee that these solutions $w_t$ and $g_t \coloneqq e^{2w_t} g_0$ also have positive scalar curvature. This shows that $S$ is open.

Now we prove that $S$ is closed. Suppose $t_n \in S$, and $t_n \rightarrow t'$. Let $w_{t_n}$ denote the corresponding solutions of \ref{eq:el}. By Proposition \ref{lemma:weak}, there exists $w_{t'} \in W^{2,2}(M)$, such that $w_{t_n} \rightharpoonup w_{t'}$ in $W^{2,2}(M)$. It is easy to check that $w_{t'}$ is a weak solution to the \ref{eq:el} for $t = t'$. The regularity theorem of \cite{UV} shows $w_{t'}$ is smooth. Next, we prove that $R_{t'} >0 $.
Since the scalar curvature $R_{t_n}$ is always positive, for any non-negative smooth test function $\phi$, the following inequality holds:
\begin{equation*}
\int \phi (\Delta_0 w_{t_n} + |\nabla_0 w_{t_n} |^2 )  = \int \frac{1}{6}(R_0 - R_{t_n} e^{2w_{t_n}} ) \phi \le \int \frac{1}{6}R_0 \phi.
\end{equation*}
As $w_{t_n} \rightharpoonup w_{t'}$, we also have 
\begin{equation*}
\int \frac{1}{6}(R_0 - R_{t'} e^{2w_{t'}} )\phi = \int \phi (\Delta_0 w_{t'} + |\nabla_0 w_{t'} |^2 )  \le \int \frac{1}{6}R_0 \phi
\end{equation*}
Thus $R_{t'}$ is non-negative in weak sense which also implies that $R_{t'}$ is non-negative pointwisely.  If $R_{t'}$ achieves 0 at some point, it is a contradiction by the strong maximum principle applied to the equation $-\Delta_t R_t +\frac{t}{4} R_t^2  = 3t |E_t|^2-12\gamma(t) |\eta|^2_t > 0$.  Hence $t' \in S$ and $S$ is closed. 
\end{proof}

\section{Compactness of solutions : closed manifolds}
Let $(M, g)$ be a closed compact Riemannian manifold with $k_p, Y(M, [g])>0$. By Proposition \ref{lemma:QR}, We can deform the metric so that both the $Q$-curvature and the scalar curvature are positive. Hence, without loss of generality, we may assume that the background metric satiesfies $Q, R> 0$ to prove Theorem \ref{thm:main1}. Also, we normalize the volume of the background metric to be 1.

We will consider the following 1-parameter family of fourth order equations with constraint for  $t \in [0,1]$.
\begin{equation}\label{eq:ls}\tag*{($\ast \ast$)$_{t}$}
\begin{cases}
\,Q_t = tk_p + (1-t)Q e^{-4w_t} \iff P_4 w_t = 2tk_p e^{4w_t} - 2tQ\\
\, \text{with $R_t>0$, and normalization $Vol(M, g_t) = \int_M e^{4w_t}dV =1$}
\end{cases}
\end{equation}
where $g_t \coloneqq e^{2w_t} g$ and $R_t$, $Q_t$ denote the scalar curvature and the $Q$-curvature of the metric $g_t$, respectively. Observe that 
$$\int_M Q_t dV_t = \int_M (tk_p + (1-t) Q e^{-4w_t} ) dV_t = tk_p + (1-t) \int_M Q dV = k_p$$
which implies that the equation is consistent wth the fact that  $k_p$ is a conformal invariant.
In this section and the next section,  $w_t$ will denote a solution of the elliptic PDE \ref{eq:ls} instead of the equation \ref{eq:el}.

Our goal in this section is to establish the compactness of the set 
\begin{equation*}
\{ w_t \, | \, \text{$w_t$ is a solution of the equation \ref{eq:ls} for some  $t \in [0,1]$ }\}
\end{equation*}
 in $C^{4, \alpha}(M)$-topology for any $0< \alpha < 1$. This compactness result will be used in the next section to prove that the equation \ref{eq:ls} indeed has a solution at $t=1 $ while we apply Leray-Schauder degree theory to the one-parameter of equations \ref{eq:ls} for $0 \le t \le 1$. Since $P_4$ is a non-negative operator, $w_0 \equiv 0$ is the only solution to the equation \ref{eq:ls} at $t = 0$, and $g_0 = g$.

We first establish some a-priori estimates for the equation \ref{eq:ls}. Through out this section we will denote constants independent of $w_t$ and $t$ by $C$. 
\begin{proposition} \label{lemma:localestimate}
Let $w_t$ be a solution of the PDE \ref{eq:ls} for some $0\le t \le 1$. Let $p \in M$ and $B_r(p)$ be a geodesic ball centered at $p$ with radius $r$.
Then, $\int_M |\nabla w_t|^2, \int_M (w_t-\overline{w}_t)^2 < C$,  and $\fint_{B_r(p)} |\nabla w_t|^2 < \frac{C}{r^2}$ for sufficiently small $r$, where $C$ is a constant independent of $t$ and $r$, and $\overline{w}_t \coloneqq \fint w_t dV$.
\end{proposition}
\begin{proof}
We have $\Delta w_t + |\nabla w_t|^2 \le \frac{1}{6} R$ from the scalar curvature equation. If we integrate both sides and apply \Poincare 's inequality, we get the first two inequalities.

For the second inequality, suppose $r$ is sufficiently smaller than the injectivity radius. We multiply a cut-off function $\eta_r^2$ on the scalar curvature equation. $\eta_r$ is a smooth test function satisfying $\eta_r \equiv 1$ on $B_r(p)$, $\eta_r \equiv 0$ on  $B_{2r}(p)$, and $|\nabla \eta_r | \le \frac{C}{r}$ for some constant $C$.
We have the following estimate
\begin{align*}
\int_{B_{2r}(p)} \eta_r^2 |\nabla w_t|^2&\le \int_M \frac{1}{6} R \eta_r^2 + 2\int_{B_{2r}(p)} \eta_r \nabla \eta_r \cdot \nabla w_t\\
&\le \int_M \frac{1}{6} R \eta_r^2 + \frac{1}{2} \int_{B_{2r}(p)} \eta_r^2 |\nabla w_t|^2 + 2\int_{B_{2r}(p)} |\nabla \eta_r|^2 \\
&\le \frac{1}{2} \int_{B_{2r}(p)} \eta_r^2 |\nabla w_t|^2 + C|R|_{L^\infty}r^4 + C r^2.
\end{align*}
This gives us the desired estimate.
\end{proof}

Next, we prove the following energy estimate for solutions of fourth order PDEs having bi-Laplacian as a leading-order term.
\begin{proposition}{(Energy estimate)} \label{lemma:energyestimate}
 Suppose $w$ is a weak solution to a fourth order PDE\, $\Delta^2 w + \delta (A ) dw +f = 0$ where $A$ is a smooth symmetric 2-tensor and $f$ is a function in $L^2(M)$. In other words, $$
\int_M \Delta w \Delta \phi + A(\nabla w, \nabla \phi) +f\phi = 0
$$  
for every $\phi \in W^{2,2}(M)$.  Then, for all sufficiently small $r>0$, $$
||\nabla^2 w ||_{L^2(B_{r})} \le  C_r (||w ||_{W^{1,2} ( B_{2r})} + ||f||_{L^2(B_{2r})})
$$
where $B_r$ and $B_{2r}$ are two concentric geodesic balls and $C_r$ is a constant depending on $r$ and $||A||_{L^{\infty}(B_{2r})}$.
\end{proposition}

\begin{proof}
We test with $\phi = \eta_r^2 w$. $\eta_r$ is a cut-off function defined in the proof of Lemma \ref{lemma:localestimate} with an additional property $|\nabla^2 \eta_r | \le \frac{C}{r^2}$. We have the following series of integral identities,
\begin{align*}
\int_M \Delta w \Delta(\eta_r^2 w )  =& \int_M \Delta w [\eta_r^2 \Delta w + 4\eta_r \nabla \eta_r \cdot w + (2\eta_r \Delta \eta_r + 2|\nabla \eta_r|^2 ) w] \\
 = &\int_M (\Delta (\eta_r w))^2- 4|\nabla \eta_r \cdot \nabla w|^2 - (\Delta \eta_r \cdot w)^2 + 2|\nabla \eta_r|^2 \Delta w \cdot w \\
& -4 \nabla \eta_r \cdot \nabla w \Delta \eta_r w \\
= & \int_M (\Delta(\eta_r w ))^2 - 2|\nabla \eta_r|^2 |\nabla w |^2 - 2 (\nabla |\nabla \eta_r|^2 \cdot \nabla w)w \\
& - 4|\nabla \eta_r \cdot \nabla w|^2 - (\Delta \eta_r \cdot w )^2 - 4(\nabla \eta_r \cdot \nabla w) \Delta \eta_r\ \cdot w \\
\ge & \int_{B_{2r}} (\Delta(\eta_r w ))^2 - C\int_{B_{2r}}  (\frac{|\nabla w|^2}{r^2} + \frac{w^2}{r^4})
\end{align*}
where at the last line, we use Young's inequality. For the lower order terms, we follow the standard argument.
\begin{align*}
\int_M A(\nabla w, \nabla ( \eta_r^2 w)) &  = \int_M \eta_r^2 A(\nabla w, \nabla w) +2 w \eta_r A(\nabla w, \nabla \eta_r) \\
 & \le C \int_{B_{2r}} ||A||_{L^\infty(B_{2r})}(|\nabla w|^2  + \frac{w^2}{r^2})
\end{align*}
\begin{align*}
\int_M f \eta_r^2 w & \ge - \int_{B_{2r}} |f| |w| \ge - ||f||_{L^2(B_{2r})} ||w||_{L^2(B_{2r})}
\end{align*}
As $\nabla (\eta_r w) |_{\partial B_{2r}} = 0$, integral by parts formula gives us $\int_{B_{2r}} (\Delta ( \eta_r w))^2 = \int_{B_{2r}} |\nabla^2 (\eta_r w)|^2 \ge \int_{B_r} |\nabla^2 w|^2$.

This completes the proof.

\end{proof}

Following lemma establishes an a-priori estimate for a strong solution of PDEs considered in Proposition \ref{lemma:localestimate}.
\begin{lemma} \label{lemma:strongestimate}
Suppose $w \in W^{4,2}(M)$ is a strong solution to the equation $\Delta^2 w + \delta (A ) dw +f = 0$ where $A$ is a smooth 2-tensor and $f \in L^2(M)$. Then, for  for sufficiently small $r$,  
$$||\nabla^2 w ||_{W^{2, 2}(B_{r})} \le  C_r (||f||_{L^2(B_{2r})} + ||w ||_{W^{1,2} ( B_{2r})}).$$
where $B_r$ and $B_{2r}$ are two concentric geodesic balls and $C_r$ is a constant depending on $r$ and $||A||_{L^{\infty}(B_{2r})}, ||\nabla A||_{L^{\infty}(B_{2r})}$.
\end{lemma}

\begin{proof}
We start from the identity $\Delta (\Delta w )  = - f + A_{ij}w_{ij} + A_{ij;i}w_j$
By the $L^2$- theory of second order elliptic PDE, and Proposition \ref{lemma:localestimate},
\begin{align*}
|| \Delta w ||_{W^{2,2}(B_r)}  \le & \, C_r (||f||_{L^2(B_{2r})} + ||A||_{L^{\infty}(B_{2r})} ||\nabla^2 w|||_{L^2 ( B_{2r})} + ||\nabla A||_{L^{\infty} ( B_{2r})}||\nabla w ||_{L^2 ( B_{2r})} + || \Delta w ||_{L^2(B_r)}  )\\
\le & C_r(||f||_{L^2(B_{2r})} + ||w||_{W^{2,2}(B_{2r})})   \\
\le & C_r (||f||_{L^2(B_{4r})} + ||w||_{L^{2}(B_{4r})}).
\end{align*}

Applying integral by parts forumla gives us the full control of $W^{4,2}(B_{r})$-norm of $w$.
\end{proof}

Now we are ready to prove the following compactness result for solutions of \ref{eq:ls}. Let $0<\alpha<1$.

\begin{theorem} \label{theorem:compactness}
There exists $C_{\alpha}>0$ such that $||w_t||_{C^{4, \alpha}} < C_{\alpha}$ for every solution $w_t$ of the equation \ref{eq:ls} where $C_\alpha$ independent of $0 \le t \le 1$.
\end{theorem}

\begin{proof}
First we claim that it suffices to show $\sup w_t < C$ to prove the theorem. Assume $\sup w_t < C$. Then the right-hand side of \ref{eq:ls} is bounded above pointwisely. Note that $w_t-\overline{w}_t$ satiesfy the same equation \ref{eq:ls} with different normalization. We apply Lemma \ref{lemma:strongestimate} to $w_t-\overline{w}_t$ This gives us $||w_t-\overline{w}_t||_{W^{4,2}(M)}< C + C ||w-\overline{w}_t||_{W^{1,2}(M) }$. Lemma \ref{lemma:strongestimate} is a local estimate, but we can patch those estimates to get a global estimate. Proposition \ref{lemma:localestimate} gives us a uniform upper bounded for $ ||w-\overline{w}_t||_{W^{1,2}(M) }$. We need to prove that $\overline{w}_t$ is uniformly bounded. $w_t$ is bounded above as $\sup w_t < C$. From the Moser-Trudinger's inequality \cite[Theorem 1.2]{CY},

\begin{equation*}
k_p \log{\fint e^{4(w_t-\overline{w}_t)}} \le \int w_t P w_t  =  \int (w_t-\overline{w}_t) P_4 (w_t-\overline{w}_t) \le C.
\end{equation*} 
As $\int e^{4w_t} = 1$, we have $0 \le C + 4 k_p\log{\overline{w}_t}$ from the above inequality. Hence we have a uniform bound of $||w_t||_{W^{4,2}(M)}$. By Sobolev mbedding theorem, $||\nabla^2 w||_{L^p(M)} < C_p$ for any $2<p<\infty$. Applying $L^p$ estimate to $w_t$ as we applied $L^2$-estimate in Lemma \ref{lemma:strongestimate}, we see that $||\Delta w||_{W^{2,p}(M)} < C_p$. Again by the $L^p$-estimate, $||w||_{W^{4,2}(M)} < C_p$. By the Morrey's inequality, $||w||_{C^{2,\alpha}(M)} <C_\alpha$. We repeat the argument of Lemma \ref{lemma:strongestimate} by applying the Schauder's estimate instead of the$L^2$-estimate, to see $||\Delta w ||_{C^{2, \alpha}(M)} < C_\alpha$. The Schauder's estimate gives us $||w||_{C^{4, \alpha}(M)} < C_\alpha$.

Now we prove that $\sup w_t < C$. Assume the contrary. We choose $t_n \rightarrow t_\infty$, $w_n$ solutions of \ref{eq:ls} for $t_n$, $w_n(p_n) \rightarrow + \infty$, $p_n \rightarrow p$. 

For $\delta$ smaller than the injectivity radius, we use the exponential map and the dilation map to define a normalized sequence $\tilde{w}_n$. Specifically, let $\tilde{w}_n (x) = w_{t_n}(r_n x) + \log r_n$ where  $r_n$ is chosen to satisfy $1 = w_n(p_n) + \log r_n = w_n(0) + \log{r_n}$. Obviously, $r_n \rightarrow 0$ and $\tilde{w}_n \le 1$. The PDE for $\tilde{w}_n$ defined on a Euclidean ball $B_{\frac{\delta}{r_n}}(0)$ is as follows:

\begin{equation}
P_{\tilde{g}_n} \tilde{w}_n ( x) + 2 r_n^4 t_n Q(r_n x) = 2t_nk_p e^{4\tilde{w}_n(x)}
\end{equation}
where $\tilde{g}_n $ is the rescaled metric converging to the Euclidean metric. Note that we have $\int_{B_{\frac{\delta}{r_n}}(0)} e^{4\tilde{w}_n} = \int_{B_{\delta}(x_n)} e^{4w_n} \le 1$.

Let $\rho_n = \frac{\delta}{r_n}$. Obviously, $\rho_n \rightarrow +\infty$. For a fixed $\rho>0$ choose suitably large $n$ such that $\rho_n>2\rho$. By Proposition \ref{lemma:localestimate} and the scaling argument, $\int_{B_{2\rho}(0)} |\nabla \tilde{w}_n | ^2 = \frac{\int_{B_{2\rho r_n}(0)}|\nabla w_n|^2}{r_n^2} \le C_\rho \rho^2$. From \Poincare 's inequality, we have $\int_{B_{2\rho}(0)} |\tilde{w}_n - \fint_{B_{\rho}(0)} \tilde{w}_n|^2 \le C_\rho \rho^4$. Now by applying Proposition 4.2 and Lemma \ref{lemma:strongestimate} to $\tilde{w}_n - \fint_{B_{\rho}(0)} \tilde{w}_n$, we have $||\tilde{w}_n - \fint_{B_\rho(0)} \tilde{w}_n||_{W^{4,2}(B_\rho(0))}< C_\rho$. In particular this inequality gives us an uniform $C^{1, \beta}(B_\rho(0))$-norm bound on $\tilde{w}_n - \fint_{B_{\rho}(0)} \tilde{w}_n$  for some $\beta>0$. Since $\sup \tilde{w}_n = \tilde{w}_n(0) = 1$, and $||\nabla \tilde{w}_n||_{L^{\infty}(B_\rho (0))} = ||\nabla(\tilde{w}_n - \fint_{B_{\rho}(0)} \tilde{w}_n)||_{L^{\infty}(B_\rho (0))}$ is bounded, this yields the uniform bound on $W^{4,2}(B_\rho(0))$ norm for $\tilde{w}_n$ on $B_\rho(0)$.

From the above paragraph, we immediately see that there exists $w_\infty \in W_{\mathrm{loc}}^{4,2} (\mathbb{R}^n)$ s.t. $w_n \rightarrow w_\infty$ in $W^{3,2}(B_\rho(0))$ and $w_n \rightharpoonup w_\infty$ in $W^{4,2}(B_\rho(0))$ for every $\rho>0$. It is easy to see that $w_\infty$ is a weak solution of the PDE $\Delta^2 w_\infty = 2t_\infty k_p e^{4w_\infty}$ with $\int e^{4w_\infty} \le 1$, $\sup w_\infty = 1$, or
\begin{equation*}
\Delta^2 \big(w_\infty-\frac{\log{3/(t_\infty k_p)}}{4}\big) = 6 e^{4[(w_\infty- (\frac{\log{3/(t_\infty k_p)}}{4})]}.
\end{equation*}

Let $\hat{w}_\infty = w_\infty-\frac{\log{3/(t_0k_p)}}{4}$.  As $R_\infty$ is non-negative and $\int_{\mathbb{R}^4} |Q_\infty| e^{4w_\infty}<\infty$, by Theorem \ref{theorem:normal}, $\hat{w}_\infty$ is normal.  Hence we have the following integral representation of $\hat{w}_\infty$

\begin{equation*}
\hat{w}_\infty(x) = \frac{3}{4\pi^2} \int_{\mathbb{R}^4} \log \big(\frac{|y|}{|x-y|} \big) e^{4\hat{w}_\infty(y)}dy + C_0.
\end{equation*}

By Theorem \ref{theorem:classification}, we know that the solution to the above integral equation is $\hat{w}_\infty (x) =  \log{(\frac{2\lambda}{\lambda^2 + |x|^2}})$ for some $\lambda >0$. Then, we have $\frac{8\pi^2}{3} = \int_{\mathbb{R}^4} e^{4\hat{w}_\infty} =  \int_{\mathbb{R}^4} e^{4w_\infty} \cdot (\frac{t_\infty k_p}{3})$, or $ \int_{\mathbb{R}^4} e^{4w_\infty} = \frac{8\pi^2}{t_\infty k_p} > 1$ as $t\le 1$ and $k_p <8\pi^2$. This contradicts $\int_{\mathbb{R}^4} e^{4w_\infty} \le 1$.

\end{proof}

\begin{remark}
The conclusion of Theorem \ref{theorem:compactness} remains valid under the normalization $\int w_t dV= 0$. This can be easily observed by the shifting solutions by appropriate constants.
\end{remark}

\section{Proof of Theorem \ref{thm:main1}}

\begin{theorem}\label{thm:compactthm}
Let $(M, g)$ be a four-dimensional compact Riemannian manifold with conformal invariants $k_p$ and $Y(M, [g])$ positive. Then there exists a conformal deformation $w$ such that $Q_w$ is constant, and $R_w>0$.
\end{theorem}

\begin{proof}
By Proposition \ref{lemma:QR}, we may assume that the background metric satisfies $Q, R>0$. We will apply Leray-Schauder degree theory to a 1-parameter family of equations \ref{eq:ls} but with different normalization $\int w_t dV_g = 0$. Let $\mathcal{O} = \{u \in C^{4, \alpha}(M) | R_u > 0, ||u||_{C^{4, \alpha}} <C_\alpha, \int_M u dV_g = 0 \}$ where $R_u$ denotes the scalar curvature of the metric $g_u \coloneqq e^{2u} g$, and $C_\alpha$ is the constant appearing in the remark after Theorem \ref{theorem:compactness}. $\mathcal{O}$ is a bounded open subset of the Banach space $C_0^{4, \alpha}(M) \coloneqq \{ f \in C^{4, \alpha}(M) | \int f = 0 \}$.

We define operators $F_t : C_0^{4, \alpha}(M) \rightarrow C_0^{4, \alpha}(M)$ by $F_t (u) = 2tP^{-1}(k_p e^{4u} - Q)$ where $t$ runs from 0 to 1.  $F_t$ is a continuous 1-parameter family of compact operators by the regularity argument discussed in the proof of Theorem  \ref{theorem:compactness}. Solving \ref{eq:ls} is equivalent to solving equations $(Id - F_t)(w_t) = 0$. At $t = 0$, $F_0 \equiv 0$ and $w_0 \equiv 0$ is the unique solution. Thus, the degree of $(Id-F_t)$ in $\overline{\mathcal{O}}$ at the point $0 \in  C_0^{4, \alpha}(M)$ is 1. 

The degree of maps $(Id-F_t)$ in $\overline{\mathcal{O}}$ at the point $0$ is well-defined and homotopy invariant if $0\notin (Id-F_t)(\partial \mathcal{O})$. Suppose $w_t \in \partial \mathcal{O}$ is a solution to the equation $(Id-F_t)(w_t)=0$. $w_t \in \partial \mathcal{O}$ implies either $ ||w_t||_{C^{4, \alpha}} = C_\alpha$, or $R_{w_t} \ge 0$ and $R_{w_t}(p_0) = 0$ for some $p_0 \in M$. The first case is obviously a contradiction to the Theorem \ref{theorem:compactness}. The second case contradicts the strong maximum principle applied to $Q_{w_t}>0$. Hence the degree of a map $(Id-F_1)$ in $\overline{\mathcal{O}}$ at the point $0$ is also 1, and there exists a solution to the equation \ref{eq:ls} at $t=1$.

\end{proof}

In order to apply the method of continuity to prove Theorem \ref{thm:main1}, we need to show that the kernel of the linearized operator $\tilde{L}_t$ of \ref{eq:ls} is $\{ 0 \}$. For  smooth $\phi$ with $\int_M \phi dV_t =0$, $\tilde{L}_t$  at a solution $w_t$ is given by
$$\tilde{L}_t (\phi) = (P_4)_t \phi -8tk_p \phi.$$
As $\tilde{L}_t$ is non-negative operator at $t = 0$, a reasonable approach is to prove that $\tilde{L}_t$ is non-negative for $t \in [0, 1]$. This is equivalent to showing that the first non-zero eigenvalue of $\tilde{L}_t$ is greater or equal to $8tk_p$. Note that  $Q_t = tk_p + (1-t)Q e^{-4w_t} \ge tk_p$. Hence one would try to estimate a lower bound on the first eigenvalue of when the $Q$-curvature is bounded below by a positive constant, and the scalar curvature is positive.

When $(M, g)$ is a 2-dimensional Riemannian surface, the Gaussian curvature is a curvature quantity analogous to the $Q$-curvature, and such estimate is true by Lichnerowicz and Obata's theorem. For $\sigma_2$-curvature, which is another type of conformal curvature that is widely studied, there is a similar result by Gursky and Streets \cite{GS}. 

For a four-dimensional Riemannian manifold $(M^4, g)$, we denote the Schouten tensor by $A = \frac{1}{2}(Ric - \frac{1}{6} Rg)$. Then the $\sigma_2$-curvature is defined as $\sigma_2(A) = -\frac{1}{2} |E|^2 + \frac{1}{24} R^2$. Observe that $Q = -\frac{1}{12} \Delta R + \frac{1}{2} \sigma_2(A)$.
\begin{proposition}\cite[Corollary 3.15]{GS}\label{lemma:gs}
Let $(M^4, g)$ be a closed Riemannian manifold such that the scalar curvature $R$ and the $\sigma_2$-curvature is positive.  Given $\phi \in C^{\infty}(M)$ such that $\int_M \phi dV = 0$, then
$$ \int_M \frac{1}{\sigma_2 (A_g)} T_1(A_g)^{ij} \nabla_i \phi \nabla_j \phi \ge 4 \int_M \phi^2 dV_g.$$
where $T_1$ is the Newton transform. The equality holds if and only if $\phi \equiv 0$ or $(M^4, g)$ is isometric the round sphere.
\end{proposition}

The above inequality is inspired by Andrews' inequality, which is proved in his unpublished work.
\begin{proposition}\cite[pg. 517]{CLN}
Let $(M^m, g)$ be a closed riemannian manifold with positive Ricci curvature. Given $\phi \in C^\infty (M)$ such that $\int_M \phi dV = 0$, then
$$\frac{m}{m-1} \int_M \phi^2 dV \le \int_M (Ric^{-1} ) ^{ij} \nabla_i \phi \nabla_j \phi.$$
\end{proposition}

One may consider an inequality similar to Andrews' inequality for a four-dimensional Riemannian manifold $(M,g)$ with positive scalar curvature and $Q$-curvature.
\begin{equation}\label{lemma:ben1}
\int P_4 \phi \cdot \phi \ge 8\int Q\phi^2 \text{ or } \int (\frac{P_4 \phi \cdot \phi}{Q}) \ge 8\int \phi^2
\end{equation}
for smooth $\phi$ such that $\int_M \phi dV = 0$.
While those two inequalities are true if $(M, g)$ is an Einstein manifold, they are not true in general.

\begin{exmp}
(\ref{lemma:ben1}) is false for some manifold $(S^4, g)$ where $g$ is a perturbed metric of the standard metric $g_c$.
\end{exmp}

\begin{proof}
We embed $S^4$ in $\mathbb{R}^5$ as usual. Let $x_1$ be a coordinate function. Recall that $f_1 = x_1$ and $f_2 = 5x_1^2 -1$ are the first and second eigenfunctions of the Laplace-Beltrami operator ,respectively.  Let $\overline{g} = e^{2f_2t} g_c$ for some sufficiently small $t$.
We will show that for some small $t$, $\int \overline{P}_4 f_1 \cdot f_1 dVol_{\overline{g}} < 8 \int \overline{Q} f_1^2 dVol_{\overline{g}}$.

From the conformal covariance of the Paneitz operator,
\begin{equation*}
\int \overline{P}_4 f_1 \cdot f_1 dVol_{\overline{g}} = \int P_4f_1 \cdot f_1 dVol_{g_c} = 8 Q\int f_1^2 dVol_{g_c}
\end{equation*}
\begin{equation*}
8\int \overline{Q} f_1^2 dVol_{\overline{g}} = \int (4t P_4 f_2 + 8Q) f_1^2 dVol_{g_c}
\end{equation*}
Note that $\int f_2 dVol_{\overline{g}} = \int f_1 e^{4tf_2} dVol_{g_c} = 0 $ as $f_1$ is an odd function while $f_2$ is an even function.
Now we compute :
\begin{align*}
\int \overline{P}_4 f_1 \cdot f_1 dVol_{\overline{g}} - 8 \int \overline{Q} f_1^2 dVol_{\overline{g}} & = 8Q\int f_1^2 dVol_{g_c} - \int(4t P_4 f_2 + 8Q) f_1^2 dVol_{g_c} \\
& = -4t \int P_4 f_2 \cdot f_1^2 dVol_{g_c} \\
& = -4t \int P_4 (5x_1^2 -1 ) \cdot x_1^2 dVol_{g_c}  \\
& = -\frac{4t}{5} \int P_4 (5x_1^2 -1 ) \cdot (5x_1^2 -1 ) \\
\end{align*}
The value of the last line is positive as $f_2$ is a eigenfuction of the laplacian. If we choose $t$ to be positive, the assertion is proved. Also note that we can also perturb $\overline{g}$ further to make it stay out of the conformal class $[g_c]$. The same method gives a counter example to the second inequality of (\ref{lemma:ben1}).

\end{proof}

\begin{remark}
If one can prove either of (\ref{lemma:ben1}) for $(M, g)$ with positive scalar curvautre and constant positive $Q$-curvature, then the solution to \ref{eq:ls} at $t=1$ is unique by degree theory.
\end{remark}


\part{Manifolds with umbilic boundary}
\section{Preliminaries}

Throughout Sections 6 and 7, we consider critical points of linear combinations of some functionals that were studied in \cite{CQ1} and \cite{CQ2} in the case of  manifolds with boundary. Let $(M, \partial M, g_0)$ be a four-dimensional Riemannian manifold with a totally geodesic boundary. Note that we can always conformally deform a manifold with an umbilic boundary to a manifold with a totally geodesic boundary.

We assume $k_{(P_4, P_3)} = \int Q + \oint T > 0 $. and $Y(M, \partial M, [g_0]>0$. By \cite[Lemma 5.2]{CN}, we have $k_{(P_4, P_3)} \le 4\pi^2$ and $k_p =4\pi^2$ if and only if $M$ is conformally equivalent to $S^4_+$ with the standard metric. Therefore, we may assume $k_{(P_4, P_3)} < 4\pi^2$. In addition, by \cite[Theorem 6.1]{E}, we can assume that the bacground metric $R_0$ is the boundary Yamabe metric and that the boundary is totally geodesic. In particular, $R_0$ is constant.
For manifolds with a totally geodesic boundary, we have simplified expressions for the functionals considered in \cite{CQ1} and \cite{CQ2}. The expressions (\ref{eq:p3operator}) and (\ref{eq:tcurvature}) are simplified to $P_3w = \frac{1}{2} \frac{\partial \Delta w}{\partial n}$ for $\frac{\partial w}{\partial n} = 0$, and $T = -\frac{1}{12}\frac{\partial R}{\partial n}$.

Let $\eta_1$ be a fixed (2,0) tensor defined on $M$ with $|\eta_1| >0$ everywhere. In addition, let $\eta_2$ be a fixed $(2, 0)$-tensor defined on $\partial M$ with $|\eta_2|_{g_0|_{\partial M}}>0$ everywhere on $\partial M$.

\begin{align*}\label{eq:bfunctional}
I[w] & = \int 4|\eta_1|_0^2 wdV_0 - (\int |\eta_1|_0^2 dV_0) \log{\fint_M e^{4w} dV_0}\\ 
I_b[w] & = \oint 3|\eta_2|_0^{\frac{3}{2}} w dV_0 - (\oint |\eta_2|_0^{\frac{3}{2}} dV_0 ) \log{\fint_{\partial M} e^{3w} dV_0} \\
II_b[w] & =  \int \big( w P_4 w + 4Q_0 w\big) dV_0  + \oint_{\partial M} \big(2w P_3 w +  4T_0 w \big) dS_0 \\
& = \int \big[(\Delta_0 w)^2 + \frac{2}{3}R_0 |\nabla_0 w|^2 - 2 Ric_0 (\nabla_0 w, \nabla_0 w) +4 Q_0 w \big] dV_0 + \oint_{\partial M} 4T_0w dS_0\\
III[w] & = \frac{1}{36}(\int R_w^2 dV_w -\int R_0^2 dV_0) = \int \big[(\Delta_0w + |\nabla_0 w|^2)^2 - \frac{1}{3}R_0 | \nabla_0 w|^2 \big]dV_0
\end{align*}

Here, we have used the same notations for $I[w], III[w]$ because the expressions are the same as in the closed manifold case.
Also note that $T_0 \equiv 0$ as $R_0$ is constant.
\bigskip

We recall the following useful lemma for our computations from now on. Note that $R_0$ is constant in our case.

\begin{lemma}(\cite[Lemma 2.10]{CN}) \label{lemma:boundarygeo}
Let $(M, \partial, g_0)$ be a Riemannian manifold with totally geodesic boundary. If $u$ is a $C^2$ function with $\frac{\partial u}{\partial n_0} = 0$, then $\frac{\partial |\nabla_0 u|^2}{\partial n_0} = 0$, and $Ric_0(n_0, \nabla_0 u) = 0$ on the boundary $\partial M$.
\end{lemma}

In this section and the next section, we are interested in the following two problems as a preliminary step for proving Theorem \ref{thm:main2}.

\bigskip
Problem 1: We would like to find a conformal deformation $w_1$ such that 

$\begin{cases} 
Q_{w_1}>0 \text{ on $M$}\\ 
R_{w_1}>0 \text{ on $M$} \\
T_{w_1}=0 \iff \frac{\partial \Delta_0 w_1}{\partial n_0} = 0 \text{ on $\partial M$}\\
H_{w_1}=0 \iff \frac{\partial w_1}{\partial n_0} = 0 \text{ on $\partial M$}.
\end{cases}$

Our strategy is to find a one-parameter family of critical points $w_t$ of functionals 
\begin{equation*} \label{eq:b1functional}
F^1_t[w] \coloneqq \gamma_1(t) I[w] + t( II_b[w]- k_{(P_4, P_3)}) \log{\fint e^{4w}dv_0}) + (1-t) III[w]
\end{equation*}
defined on the set $\{w\in W^{2,2}(M)| \frac{\partial w}{\partial n_0} \equiv 0 \}$, using the method of continuity. Here,  $\gamma_1(t) = \frac{-tk_p}{\int |\eta_1|_0^2 dV_0}$.

 The Euler Lagrange equation for a critical point of the functional $F^1_t$ is
\begin{equation}\label{eq:b1el}\tag*{($\star$)$_{t}$}
\begin{cases} \
\gamma_1(t) |\eta|_t^2 + tQ_t + \frac{(1-t)}{12}(-\Delta_t R_t ) = 0 \\
T_t = 0 \\
H_t = 0 
\end{cases}
\end{equation}
Obviously $H_t = 0$ is automatically satisfied since $ \frac{\partial w}{\partial n_0} = 0$.
The weak formulation of the above equation is

\begin{multline} \label{eq:b1weak}
\int \Delta_0 w_t \Delta_0 \phi +  (1-t)\big[2(\nabla_0 w_t \cdot \nabla_0 \phi)(\Delta_0 w_t + |\nabla_0 w_t|^2 ) + |\nabla_0 w_t|^2 \Delta_0 \phi \big]\\
= \int \big[(\frac{1}{3}-t) R_0g_0 + 2t Ric_0\big](\nabla_0 w_t, \nabla_0 \phi) - \int (2tQ_0 + 2 \gamma_1(t) |\eta_1|^2_0)\phi
\end{multline}
for any $\phi \in W^{2,2}(M)$ with $\frac{\partial \phi}{\partial n_0} \equiv 0 $.

\bigskip

Problem 2: We would like to find a conformal deformation $w_2$ such that 

$\begin{cases} 
Q_{w_2} =0 \iff (P_4)_0 +2Q_0 = 0  \text{ on $M$}\\ 
R_{w_2}>0 \text{ on $M$}\\
T_{w_2}>0 \text{ on $\partial M$}\\
H_{w_2}=0 \iff \frac{\partial w_2}{\partial n_0} = 0 \text{ on $\partial M$}.
\end{cases}$

Likewise, our strategy is to find a 1-parameter family of critical points $w_t$ of functionals 

\begin{equation*} \label{eq:b2functional}
F^2_t [w] \coloneqq \gamma_1(t) I_b[w] + t( II_b[w] - \frac{4}{3}k_{(P_4, P_3)}\log{\fint_{\partial M} e^{3w}})+ (1-t) III[w]
\end{equation*}

defined on the set $\{w\in W^{2,2}(M)| \frac{\partial w}{\partial n_0} \equiv 0 \}$. Here,  $\gamma_2(t) = \frac{-4tk_{(P_4,P_3)}}{3\int |\eta_2|_0^{3/2} dV_0}$.

The Euler Lagrange equation for  for a critical point of the functional $F^2_t$ is
\begin{equation}\label{eq:b2el}\tag*{($\dagger$)$_{t}$}
\begin{cases} 
 tQ_t + \frac{(1-t)}{12}(-\Delta_t R_t ) = 0 \\
T_t = -\frac{3}{4}\gamma_2(t) |\eta_2|_t^{\frac{3}{2}} \\
H_t = 0 
\end{cases}
\end{equation}
The weak formulation of the above equation is
\begin{multline} \label{eq:b2weak}
\int \Delta_0 w_t \Delta_0 \phi +  (1-t)\big[2(\nabla_0 w_t \cdot \nabla_0 \phi)(\Delta_0 w_t + |\nabla_0 w_t|^2 ) + |\nabla_0 w_t|^2 \Delta_0 \phi \big] \\
=\int \big[(\frac{1}{3}-t) R_0g_0 + 2t Ric_0\big](\nabla_0 w_t, \nabla_0 \phi) - \int 2tQ_0 \phi - \frac{3}{2} \gamma_2(t) \oint |\eta_2|_0^{\frac{3}{2}} \phi
\end{multline}
for any $\phi \in W^{2,2}(M)$ with $\frac{\partial \phi}{\partial n_0} \equiv 0 $.

We will vary the paramter $t$ from $\frac{2}{3}$ to $1$. Let $t_0 = \frac{2}{3}$.
\begin{lemma} \label{lemma:bstarting}
There exist minimizers for both $F^1_{t_0}$ and $F^2_{t_0}$. For a minimizer $w_{t_0}$ of $F^1_{t_0}$ or $F^2_{t_0}$, the scalar curvature $R_{t_0}$ is positive.
\end{lemma}

\begin{proof}

By \cite[Theorem 1.9]{CN}, $(P_4)_0$ is non-negative operator and $\ker (P_4)_0 \cong \mathbb{R}$. Therefore, we can find minimizers of $F^1_{\frac{2}{3}}$ or $F^2_{\frac{2}{3}}$ as in Lemma \ref{lemma:starting}. These minimizers are smooth on the interior by \cite{UV}. On boundary points, the same method as in \cite{UV} is applied, except that we have to use $W^{2, p}$-estimates for Neumann boundary problems.

By \cite[Lemma 1.1]{E}, and \cite[Proposition 1.3]{E}, the first non zero-Neumann eigenvalue $\lambda$ of the linear elliptic operator $(-\Delta_{t_0} + \frac{1}{6}R_{t_0})$ is positive. The eigenfunction $f$ corresponding to $\lambda$, is positive on $\overline{M}$. We have $(-\Delta_{t_0} + \frac{1}{6} R_{t_0}) f = \lambda f$, $\frac{\partial f}{\partial n_{t_0}} = 0$.

If $w_{t_0}$ is a solution of either problem 1 or problem 2, the $Q$-curvature equation gives us $(-\Delta_{t_0} + \frac{1}{6})R_{t_0} \ge 0$ and the $T$ curvature equation gives us $\frac{\partial R_{t_0}}{\partial n_{t_0}} \le 0 $.

Choose $p_0$ such that $\min_{\overline{M}} \frac{R_{t_0}}{f} = \frac{R_{t_0}}{f}(p_0)$. If $p_0$ is on the interior, we compute $\Delta_{t_0}(\frac{R_{t_0}}{f})(p_0) =[\frac{\Delta_{t_0} R_{t_0} - \frac{1}{6} R_{t_0}^2}{f} + \frac{\lambda R_{t_0}}{f}](p_0) \le \frac{\lambda R_{t_0}}{f}(p_0)$. As $f>0$, we get $R_{t_0}(p_0) \ge 0$. If $R_{t_0}(p_0) =0$, then by the strong maximum principle $R_{t_0} \equiv 0$, which contradicts $Y(M, \partial M, [g_0])>0$. Hence $R_{t_0} > 0$ in this case. Next, suppose $p_0$ is on the boundary. If $\frac{R_{t_0}}{f}(p_0) \le 0$, then by Hopf's lemma we have $\frac{\partial(R_{t_0}/f)}{\partial n_{t_0}} > 0$. This contradicts the $T$-curvature equation.

\end{proof}

Next, we observe that any critical point $w_t$ of $F_t^1$ or $F_t^2$  satisfies $||w_t||_{W^{2,2}} \le C$ for some constant independent of $t$ for $\frac{2}{3} \le t \le 1$. This fact will be used to prove closedness part when we apply the method of continuity in the next section.

\begin{proposition}\label{lemma:bclosedness}
Let $w_t$ be a critical point of $F^1_t$ or $F^2_t$ with normalization $\int_M w_t = 0$ for $F^1_t$ and $\oint_{\partial M} w_t = 0 $ for $F^2_t$. Assume $R_t>0$ for both cases. Then $||w_t||_{W^{2,2}(M)} \le C$ for some $C$ independent of $t$ and $w_t$.
\end{proposition}

\begin{proof}
The proof is similar to that of Proposition \ref{lemma:weak}.
\end{proof}

\section{The method of continuity}


In this section we prove the positivity of linearized operators for equations considered in Section 6. The basic strategy follows that of the Section 3 with only minor differences.

\begin{lemma} \label{lemma:bopenness}
Let $\frac{2}{3} \le t \le 1$. Linearized problems for \ref{eq:b1el} and \ref{eq:b2el} only have trivial solution under normalization $\int_M w_t = 0$ or $\oint_{\partial M} w_t = 0 $ when $R_t>0$.
\end{lemma}

\begin{proof}
By direct computation, both linearized equations are as folllows:

$\begin{cases} 
L_t(\phi)  \coloneqq \Delta_t^2\phi + tdiv_t (\frac{2}{3} R_t g_t - 2 Ric_t )d\phi + \frac{1-t}{3} \nabla_t (R_t \nabla_t \phi) = 0 \text{ on $M$} \\ 
\frac{\partial \Delta_t \phi}{\partial n_t} = 0  \text{ on $\partial M$}\\
\frac{\partial \phi}{\partial n_t} = 0  \text{ on $\partial M$}
\end{cases}$

Let $\phi$ be a function with $\frac{\partial \Delta_t \phi}{\partial n_t} = 0$ and $\frac{\partial \phi}{\partial n_t} = 0$ on the boundary. We prove that
$$\int_M \langle L_t(\phi), \phi \rangle dV_t \ge 0.$$
Since the boundary is totally geodesic, $Ric_t(n_t, \nabla_t \phi) = 0$ and $\frac{\partial |\nabla_t \phi|^2 }{\partial n_t} = 0$ by Lemma \ref{lemma:boundarygeo}.

Denote the traceless part of the Ricci cuvature of the metric $g_t$ by $E_t$. Let $\phi \in W^{2,2}(M)$. The equation \ref{eq:el} is rewirtten as
\begin{equation}\label{eq:tracelessel}
-\Delta_t R_t +\frac{t}{4} R_t^2 - 3t |E_t|^2 = -12\gamma(t) |\eta|^2_t > 0.
\end{equation}

There are two identities in the proof of Lemma \ref{lemma:linearization} that include additional boundary integral terms.
The first identity is
\begin{align*}
\langle L_t \phi, \phi \rangle =& \int_M (\Delta_t \phi)^2 -2t \int_M E_t(\nabla_t \phi,\nabla_t \phi) + ( \frac{1}{2} t - \frac{1}{3})\int_M |\nabla_t \phi|^2 \\
& -\oint_{\partial M} \frac{\partial \Delta_t \phi}{\partial n_t} \phi + \oint_{\partial M} \Delta_t \phi \frac{\partial \phi}{\partial n_t} + \oint_{\partial M} \big( -\frac{2R_t}{3} \big) \frac{\partial \phi}{\partial n_t} +2Ric_t(\nabla_t \phi, n_t) \\
=& \int_M (\Delta_t \phi)^2 -2t \int_M E_t(\nabla_t \phi,\nabla_t \phi) + ( \frac{1}{2} t - \frac{1}{3})\int_M |\nabla_t \phi|^2.
\end{align*}

The second identity is
\begin{align*}
\int \frac{\Delta_t R_t }{R_t} |\nabla_t \phi|^2 =& \int_M - \nabla_t R_t \cdot \nabla_t(R_t^{-1}) |\nabla_t \phi|^2 - \frac{\nabla_t R_t}{R_t} \cdot \nabla_t |\nabla_t \phi|^2 - \oint_{\partial M} \frac{1}{R_t} \frac{\partial R_t}{\partial n_t}  |\nabla_t \phi|^2 \\
= & \int_M - \nabla_t R_t \cdot \nabla_t(R_t^{-1}) |\nabla_t \phi|^2 - \frac{\nabla_t R_t}{R_t} \cdot \nabla_t |\nabla_t \phi|^2 + \oint_{\partial M} 12\big( \frac{T_t}{R_t} \big)  |\nabla_t \phi|^2.
\end{align*}
The rest of the proof is exactly the same as the closed manifold case. We have the inequality
\begin{equation*}
\langle L_t \phi, \phi \rangle \ge C_1(t) \int (\Delta_t \phi)^2  + C_2(t) \int R_t|\nabla_t \phi|^2 + C_3(t) \oint_{\partial M} \big( \frac{T_t}{R_t} \big)  |\nabla_t \phi|^2.
\end{equation*}
for some positive constants $C_1(t)$, $C_2(t)$, and $C_3(t)$. As both of \ref{eq:b1el} and \ref{eq:b2el} imply $T_t \ge 0$, we obtain the result.

\end{proof}

\begin{proposition} \label{lemma:bQR}
Let $(M, \partial M, g_0)$ be a compact manifold with umbilic boundary and the conformal invariants $k_{(P_4, P_3)}$ and $Y(M, \partial M, [g_0])$ positive. Then, there are conformal deformations $w_1, w_2$ with  follwowing properties  

$\begin{cases} 
Q_{w_1}>0, R_{w_1}>0  \text{ on $M$}  \\ 
T_{w_1}= 0, H_{w_1}=0 \text{ on $\partial M$}
\end{cases}$ and  \,
$\begin{cases} 
Q_{w_2}=0, R_{w_2}>0  \text{ on $M$}  \\ 
T_{w_2}>0 , H_{w_2} =0 \text{ on $\partial M$}
\end{cases}$

\end{proposition}

\begin{proof}
The proof is identical to that of Proposition \ref{lemma:QR} using Lemma \ref{lemma:bstarting}, Proposition \ref{lemma:bclosedness}, and Lemma \ref{lemma:bopenness}, except for that we have to use Hopf maximum principle to show the scalar curvature is strictly positive.
\end{proof}

\section{Compactness of solutions : manifolds with boundary}

We aim to prove a compactness result for solutions of two one-parameter families of equations, which is a key step in establishing Theorem \ref{thm:main2}. This section is analogue of Section 4 for manifolds with boundary. Let $(M, \partial M, g)$ be a compact Riemannian manifold with totally geodesic boundary. once again, we assume $k_{(P_4, P_3)}, Y(M, \partial M, [g])>0$.

Firstly, we describe an one-parameter family of equations that will be used to find $w_1$ in Theorem \ref{thm:main2}. In this case, we conformally deform the metric $g$ into $g_1$ so that the boundary is totally geodesic, $T_1\equiv 0$, and both  $Q_1$ and $R_1$ are positive pointwisely, according to the first part of Proposition \ref{lemma:bQR}. Under these assumptions, we show that there exists an uniform $C^{4,\alpha}$ bound on solutions to the following equations: 
\begin{equation} \label{eq:bb1ls}\tag*{($\star \star$)$_{t}$}
\begin{cases} 
Q_t = t k_p e^{-4w_t} +(1-t)Q_1 \iff (P_4)_1w_t = 2tk_{(P_4, P_3)} e^{4w_t} - 2tQ_1  \\
R_t>0 \text{ on $\overline{M}$} \\
T_t=0 \iff \frac{\partial \Delta_1 w_t}{\partial n_1} = 0  \\
H_t=0 \iff \frac{\partial w_t}{\partial n_1} = 0  \\
\int_M e^{4w_t} = 1
\end{cases}
\end{equation} where $Q_t$, $R_t$, $T_t$, and $H_t$ denote the $Q$-curvature, scalar curvature, $T$-curvature, and mean curvature of the metric $g_t \coloneqq e^{2w_t} g_1$ respectively.

Secondly, we describe an one-parameter family of equations that will be used for finding $w_2$ in Theorem \ref{thm:main2}. In this case, we conformally deform the metric $g$ into $g_2$ so that the boundary is totally geodesic, $Q_2 \equiv 0$, and both of $T_2$ and $R_2$ are positive pointwisely according to the second part of Proposition \ref{lemma:bQR}. Under these assumptions, we show that there exists an uniform $C^{4,\alpha}$ bound on solutions of following equations: 
\begin{equation} \label{eq:bb2ls}\tag*{($\dagger \dagger$)$_{t}$}
\begin{cases} 
Q_t=0 \iff (P_4)_2w_t = 0  \\ 
R_t>0 \text{ on $\overline{M}$}\\
T_t =  tk_{(P_4, P_4)}e^{-3w_t} + (1-t)T_2 \iff  (P_3)_2 w_t = tk_{(P_4, P_4)}e^{3w_t}  -tT_2 \\
H_t=0 \iff \frac{\partial w}{\partial n_2} = 0 \\
\int_{\partial M} e^{3w_t} = 1
\end{cases}
\end{equation} where $Q_t$, $R_t$, $T_t$, and $H_t$ denote the $Q$-curvature, scalar curvature, $T$-curvature, and mean curvature of the metric $g_t \coloneqq e^{2w_t} g_2$ respectively.

For notational convenience, we will suppress subscripts 1 and 2 whenever the argument is independent of whether background metric is $g_1$ or $g_2$.

\begin{proposition}\label{lemma:blocalestimate}
Let $w_t$ be a solution of the PDE \ref{eq:bb1ls} or \ref{eq:bb2ls} for some $0\le t \le 1$. Let $p \in M$ and $B_r(p)$ be a geodesic ball centered at $p$ with radius $r$.
Then, $\int_M |\nabla w_t|^2, \, \int_M (w_t-\fint_{\partial M} {w}_t)^2,  \, \int_M (w_t-\fint_{\partial M} {w}_t)^2 < C$, and $\fint_{B_r(p)\cap M} |\nabla w_t|^2 < \frac{C}{r^2}$ for sufficiently small $r$, where $C$ is a constant independent of $t$ and $r$.
\end{proposition}
\begin{proof}
Taking into account of the fact $\frac{\partial w_t}{\partial n} = 0$, there is no difference in the computation with the closed manifold case Proposition \ref{lemma:localestimate}.
\end{proof}

Next, we prove the following energy estimate for solutions of Neumann boundary valued fourth order PDEs having bi-Laplacian as a leading-order term.
\begin{proposition}{(Boundary energy estimate)}\label{lemma:benergyestimate}
Suppose $w \in W^{2, 2}(M)$ is a weak solution to the following fourth-order PDE with Neumann boundary condition.
\begin{equation} \label{eq:weaksolution}
\begin{cases}
\Delta^2 w + \delta (A ) dw +f = 0 \\
\frac{\partial \Delta w } {\partial n } = - g \\
\frac{\partial w}{\partial n} = 0
\end{cases}
\end{equation}
where $f$, $g$ are  bounded functions, $A$ is a smooth symmetric 2-tensor and such that $A(\nabla w, n) = 0$(See Lemma \ref{lemma:boundarygeo}). In other words, $$
\int_M \Delta w \Delta \phi + E(\nabla w, \nabla \phi) +f\phi + \int_{\partial M} g\phi = 0
$$  
for every $\phi \in W^{2,2}(M)$ with $\frac{\partial \phi}{\partial n} \equiv =0$.  Then, for all sufficiently small $r>0$, $$
||\nabla^2 w ||_{L^2(B_{r})} \le  C_r (|| w ||_{W^{1,2} ( B_{2r})} + ||f||_{L^2(B_{2r})}  +  ||g||_{L^2(B_{2r}\cap \partial M)})  
$$
where $B_r$ and $B_{2r}$ are two concentric geodesic balls and $C_r$ is a constant depending on $r$ and $||A||_{L^{\infty}(B_{2r})}$.
\end{proposition}

\begin{proof}
We test with $\phi = \eta_r^2 w$. The only difference is that we have to choose $\eta_r$ with $\frac{\partial \eta_r}{\partial n} \equiv 0$ on $\partial M$. This can be achieved by choosing $\eta_r$ depending only on the distance from a point on the boundary. As the boundary is totally geodesic, it is easy to see that $\frac{\partial \eta_r}{\partial n} \equiv 0$ is achieved. The rest is the same as that of Proposition \ref{lemma:energyestimate}.

\end{proof}
Also, we record a basic fact regarding the boundary regularity.
\begin{lemma} \label{lemma:bstrongestimate}

Suppose $w \in W^{4,2}(M)$ is a strong solution to the equation (\ref{eq:weaksolution}) with the same conditions on $A$ and $f, g$. Then , we have
$$||\nabla^2 w ||_{W^{2, 2}(B_{r})} \le  C_r (||f||_{L^2(B_{2r})} +|| w||_{W^{1,2} ( B_{2r})} + ||g ||_{W^{1/2, 2}(B_{2r}\cap \partial M)}).$$
where $B_r$ and $B_{2r}$ are two concentric geodesic balls and $C_r$ is a constant depending on $r$ and $||A||_{L^{\infty}(B_{2r})}, ||\nabla A||_{L^{\infty}(B_{2r})}$.
\end{lemma}

\begin{proof}
The only difference with Lemma \ref{lemma:strongestimate} is that we use the $L^2$-estimate for Neumann boundary problem:
\begin{equation*}
||\nabla^2 v||_{L^2(B_{r})} \le  C_r (||\Delta v||_{L^2(B_{2r})}  + || v||_{L^2(B_{2r})} + ||\frac{\partial v}{\partial n} ||_{W^{1/2, 2}(B_{2r}\cap \partial M)}).
\end{equation*}
\end{proof}


Now we are ready to prove the following compactness results for solutions of \ref{eq:bb1ls} and \ref{eq:bb2ls}. Let $0<\alpha<1$.

\begin{theorem} \label{thm:b1compactness}
There exists $C_{\alpha}>0$ such that $||w_t||_{C^{4, \alpha}} < C_{\alpha}$ for every solution $w_t$ of the equation \ref{eq:bb1ls} where $C_\alpha$ independent of $0 \le t \le 1$.
\end{theorem}

\begin{proof}
We use the notation $g$ instead of $g_1$ for conveniece. It is enought to show that $\sup_t w_t < C$ uniformly. Suppose there exists $t_n \rightarrow t_\infty$, $w_n$ solution at $t_n$, $w_n(p_n) \rightarrow +\infty$, $p_n \rightarrow p$. Let $d_n = \mathrm{dist}(p_n, \partial M)$, and choose $r_n >0$ such that $w_n(p_n) + \log r_n = 1$.  Clearly $r_n > 0$ and $r_n \rightarrow 0$. Denote $\epsilon_n = \frac{d_n}{r_n}$ and let $\delta$ be the injectivity radius of $M$. Depending on whether $\lim d_n =0$ or $\limsup \epsilon_n < \infty$, we have separate arguments.

Case 1: $\liminf d_n >0$. This is the case when $p_n$ uniformly stays away from the boundary. Let  $d_n > \epsilon$ for some $\epsilon > 0 $.  We define $\tilde{w}_n(x) = w_{t_n}(r_n x) + \log(r_n)$  for $x \in B_{\min \{\epsilon/r_n , \delta/r_n \}}(0)$ using the exponential map. In this case, the blow up argument is the same as the closed manifold case. 

Case 2: $\lim d_n = 0$, and $\limsup \epsilon_n = \infty$. In this case, $p_n$ approaches the boundary at a relatively slow pace so we do not observe the boundary when we blow-up. We define  $\tilde{w}_n(x) = w_n(r_n x) + \log(r_n)$ for $x \in B_{\epsilon_n/2}(0)$. Again, the blow up argument is the same as the closed manifold case.

Case 3: $\lim d_n = 0$, and $\epsilon_n \rightarrow \epsilon$ for some $\epsilon \ge 0$. This is the only case we take the boundary into account after the blow-up. Let $q_n \in \partial M$ such that $\mathrm{dist}(p_n, q_n ) = \mathrm{dist}(p_n, \partial M)$. Using the exponential map centered at $q_n$, and since the boundary is totally geodesic, we can define  $\tilde{w}_n(x) = w(r_n x) + \log(r_n)$  for $x \in B_{\delta/(2r_n)}^+(0)$.  $\tilde w_n(0, \epsilon_n) = w(p_n) = 1$. By Proposition \ref{lemma:benergyestimate}, Lemma \ref{lemma:bstrongestimate}, there exists $w_\infty \in W^{4,2}_{\mathrm{loc}}(\mathbb{R}^4_+)$ s.t. $\tilde{w}_n \rightarrow w_\infty$ in $W^{3,2}(B_R^+(0))$. $\tilde{w}_n \rightharpoonup w_\infty$ in $W^{4,2}(B_R^+(0)).$

It is straightforward to observe that $w_\infty$ is a weak solution of the equation $\Delta^2 w_\infty = 2t_\infty k_{(P_4, P_3)} e^{4w_\infty}$ with boundary conditions $\frac{\partial w_\infty}{\partial y} = 0 $, $\frac{\partial \Delta w_\infty}{\partial y} = 0 $ on $\mathbb{R}^3$. Clearly, we have $\int e^{4w_\infty} \le 1 $, $\sup w_\infty = 1$.

Define $\overline{w}_\infty (x, y) = w_\infty(x, y)$ for $y \ge 0$, $\overline{w}_\infty ( x, y ) = w_\infty (x, -y)$ for $ y \le 0$. Since $\frac{\partial \Delta w_\infty}{\partial y} = 0 $ , $\overline{w}_\infty \in W^{4,2}_{loc}(\mathbb{R}^4)$. We check below that $\overline{w}_\infty$ is a weak solution of  the equation $\Delta^2 \overline{w}_\infty = 2t_\infty k_{(P_4, P_3)} e^{4\overline{w}_\infty}$ on $\mathbb{R}^4$.

\begin{align*}
\int_{\mathbb{R}^4}\overline{w}_\infty  \Delta^2 \phi =& \int_{\mathbb{R}^4_+} w_\infty  \Delta^2  \phi + w_\infty  \Delta^2  \phi(x, -y)  \\
=&  \int_{\mathbb{R}^4_+} \phi \Delta^2  w_\infty  +\phi(x, -y)  \Delta^2  w_\infty \\
&+ \int_{\mathbb{R}^3} w_\infty \frac{\partial \Delta \phi}{\partial y} + \Delta w_\infty \frac{\partial \phi}{\partial y} + w_\infty \frac{\partial \Delta \phi}{\partial y}(x, -y) + \Delta w_\infty \frac{\partial \phi}{\partial y}(x, -y)\\
= & \int_{\mathbb{R}^4_+}   2t_\infty k_{(P_4, P_3)} e^{4{w_\infty}}(\phi(x, y) + \phi(x, -y) )
= \int_{\mathbb{R}^4}2t_\infty k_{(P_4, P_3)} e^{4\overline{w}_\infty}\phi.
\end{align*}

Now the rest of the proof is now the same as in the closed manifold case. Since $\int_{\mathbb{R}^4} e^{4\overline{w}_\infty} \le 2$ and $k_{(P_4, P_3)} <4\pi^2$, the argument is applied to $\overline{w}_\infty - \frac{\log 2}{4}$.
\end{proof}


\begin{theorem} \label{thm:b2compactness}
There exists $C_{\alpha}>0$ such that $||w_t||_{C^{4, \alpha}} < C_{\alpha}$ for every solution $w_t$ of the equation \ref{eq:bb2ls} where $C_\alpha$ independent of $0 \le t \le 1$.
\end{theorem}

\begin{proof}
Again, we abuse the notation $g$ with $g_2$ for conveniece. Assume that $\sup_t w_t < C$ uniformly. We apply $L^2$-esimtate for Neumann boundary problem to $w_t - (w_t)_{\partial M}$. We have
\begin{align*}
||\nabla^2 w_t ||_{W^{2,2}(M)} \le C (||w_t - (w_t)_{\partial M}||_{W^{1,2}(M)} + ||k_{(P_4, P_3)} e^{3w_t}||_{W^{1/2,2}(\partial M)} + ||T||_{W^{1/2,2}(\partial M)} ).
\end{align*}
As $||w_t - (w_t)_{\partial M}||_{W^{1,2}(M)}$ is bounded by Lemma \ref{lemma:blocalestimate}, we have the estimate
\begin{equation*}
|| e^{3w_t}||_{W^{1/2,2}(\partial M)}\le Ce^{3M} (1+ ||w_t-(w_t)_{\partial M}||_{W^{1/2,2}(\partial M)}) \le Ce^{3M} (1+ ||w_t - (w_t)_{\partial M}||_{W^{1,2}(M)})
\end{equation*}
by the trace Sobolev inequality. Thus, we have uniform bound on $||w_t - (w_t)_{\partial M} ||_{W^{4,2}(M)}$. 
The upperbound on $(w_t)_{\partial M}$ is given by the Jensen's inequality, and the lower bound is derived from the boundary Moser-Trudinger inequality  \cite[Lemma 2.4]{N2}. Therefore, it is sufficient to show that $\sup_t w_t < C$ to prove the theorem.

Suppose there exists $t_n \rightarrow t_\infty$, $w_n$ solution for $t_n$, $w_n(p_n) \rightarrow +\infty$, $p_n \rightarrow p$. Let $d_n = \mathrm{dist}(p_n, \partial M)$, and choose $r_n >0$ such that $w_n(p_n) + \log r_n = 1$.  Clearly $r_n > 0$ and $r_n \rightarrow 0$. Denote $\epsilon_n = \frac{d_n}{r_n}$ and let $\delta$ be the injectivity radius of $M$. Case 1 and Case 2, as in the proof of Theorem \ref{thm:b1compactness}, are dealt with by the  exactly same argument. We only need to take care of Case 3: $\lim d_n = 0$, and $\epsilon_n \rightarrow \epsilon$ for some $\epsilon \ge 0$.

Let $q_n \in \partial M$ such that $\mathrm{dist}(p_n, q_n ) = \mathrm{dist}(p_n, \partial M)$. Using the exponential map centered at $q_n$, and since the boundary is totally geodesic, we can define  $\tilde{w}_n(x) = w_n(r_n x) + \log(r_n)$  for $x \in B_{\delta/(2r_n)}^+(0)$.  $\tilde{ w}_n(0, \epsilon_n) = w_n(p_n) = 1$. By Proposition \ref{lemma:benergyestimate}, Lemma \ref{lemma:bstrongestimate}, there exists $w_\infty \in W^{4,2}_{\mathrm{loc}}(\mathbb{R}^4_+)$ s.t. $\tilde{w}_n \rightarrow w_\infty$ in $W^{3,2}(B_R^+(0))$. $w_n \rightharpoonup w_\infty$ in $W^{4,2}(B_R^+(0)).$

It is easy to show that $w_\infty$ satisfies the equation $\Delta^2 w_\infty = 0$ with the boundary conditions $\frac{\partial w_\infty}{\partial y} = 0 $ and $\frac{\partial \Delta w_\infty}{\partial y} = 2t_\infty k_{(P_4, P_3)} e^{3w_\infty}  $ on the boundary $\mathbb{R}^3$. Clearly, we have $\Delta w_\infty  + |\nabla w_\infty |^2 \le 0$, $\int_{\mathbb{R}^3} e^{3w_\infty} \le 1 $, $\sup w_\infty = 1$. 

Our next goal is to prove that $w_\infty$ is normal for our $T$-curvature problem, as defined in a manner similar to Definition \ref{def:normal}. We imitate the proof of Theorem \ref{theorem:normal}.
\begin{lemma}
$w_\infty (x) = \frac{1}{2\pi^2} \int_{\mathbb{R}^3} \log{(\frac{|y|}{|x-y|})}t_\infty k_{(P_4, P_3)} e^{3w_\infty (y)} dy + C_0$.
\end{lemma}
\begin{proof}
Let $v =  \frac{1}{2\pi^2} \int_{\mathbb{R}^3} \log{(\frac{|y|}{|x-y|})}t_\infty k_{(P_4, P_3)} e^{3w_\infty (y)} dy$ and $h = w_\infty -v$. Then, we have $\Delta^2 h = 0$ on the upper half plane and $\frac{\partial h}{\partial y} = 0$ and $\frac{\partial \Delta h}{\partial y} = 0 $ on the boundary. We will show that $h$ is a constant. 

$\Delta h$ is a harmonic function with $\frac{\partial \Delta h}{\partial y} = 0 $ allowing us to consider its even reflection, which yields a smooth, entire, harmonic function. We can also extend $h$, $w_\infty$ and $v$ to be $C^2$ functions on the entire plane using even reflection. We abuse notations $h$, $w_\infty$, and $v$ to refer to these extended function. By the mean value theorem, we have
\begin{align*}
\Delta h(x_0) & = \fint_{\partial B_r(x_0)} \Delta h\\
& \le -\fint_{\partial B_r(x_0)} |\nabla w_\infty |^2 - \fint_{\partial B_r(x_0)} \Delta v \\
& \le  - C \int_{\mathbb{R}^3} \big[\fint_{S^3} \frac{1}{|r\sigma +x_0 -y|^2}\big]tk_{(P_4, P_3)} e^{3w_\infty (y)} \\
& \le -\frac{C}{\pi r^2}.
\end{align*}
If we take $r \rightarrow \infty$, we obtain $\Delta h (x_0) \le 0$ for all $x_0 \in \mathbb{R}^4$. Therefore, by Liouville's theorem, we have $ \Delta h \equiv c_0$ for some $c_0 \le 0$.

Again by applying the mean value theorem and Cauchy Schwartz inequality, we obtain
\begin{equation*}
|\nabla h|^2(x_0) \le C \fint_{\partial B_r(x_0)} |\nabla h|^2.
\end{equation*}
Next, we observe that
\begin{align*}
|\nabla h |^2(x_0) &\le 2|\nabla w_\infty |^2(x_0) + 2|\nabla v|^2(x_0)\\
& \le -2\Delta w_\infty (x_0) + 2|\nabla v|^2(x_0)\\
&\le  -2c_0 - 2\Delta v(x_0) + 2|\nabla v|^2(x_0).
\end{align*}
We handle the $\Delta v$ term as shown previously. To estimate $|\nabla v|^2(x_0)$, we have
\begin{equation*}
|\nabla v|^2(x_0) \le C(\int_{\mathbb{R}^3} \frac{1}{|x_0 - y|^2} k_{(P_4, P_3)} e^{3w_\infty (y)})(\int_{\mathbb{R}^3}k_{(P_4, P_3)} e^{3w_\infty (y)}) \le C\int_{\mathbb{R}^3} \frac{1}{|x_0 - y|^2} e^{3w_\infty (y)}.
\end{equation*}
This can also be estimated as $\Delta v$ using the same technique.  Hence, taking $r \rightarrow \infty$, we see that $|\nabla h|^2$ is bounded. Since every partial derivative of $h$ is  a harmonic function, it is contant by Liouville theorem. Therefore, $\Delta h = c_0 = 0 $, and we can conclude that $h$ is a constant.
\end{proof}

Now by the above lemma, we have
\begin{equation*}
w_\infty (x)-(\frac{\log{2/(t_\infty k_{(P_4, P_3)})}}{3}) = \frac{1}{\pi^2} \int_{\mathbb{R}^3} \log{(\frac{|y|}{|x-y|})} e^{3(w_\infty (y)-(\frac{\log{2/(t_\infty k_{(P_4, P_3)})}}{3}))} dy + C_0
\end{equation*}

Let $\hat{w}_\infty (x) = w_\infty (x)-(\frac{\log{2/(t_\infty k_{(P_4, P_3)})}}{3})$. 
The following integral equation holds.
\begin{equation}
\hat{w}_\infty (x) = \frac{1}{\pi^2} \int_{\mathbb{R}^3} \log{(\frac{|y|}{|x-y|})} e^{3\hat{w}_\infty (x)}dy  +C_0.
\end{equation}
By Theorem \ref{theorem:classification},  we have $\hat{w}_\infty (x) =  \log{(\frac{2\lambda}{\lambda^2 + |x|^2}})$ for some $\lambda >0$. Then, we see that $2\pi^2 = \int_{\mathbb{R}^3} e^{3\hat{w}_\infty} =  \int_{\mathbb{R}^4} e^{3w_\infty} \cdot (\frac{t_\infty k_{(P_4, P_3)}}{2})$, or $ \int_{\mathbb{R}^3} e^{3w_\infty } = \frac{4\pi^2}{t_\infty k_{(P_4, P_3)}} > 1$. This contradicts $\int_{\mathbb{R}^3} e^{3w_\infty } \le 1$.
\end{proof}

\section{Proof of Theorem 1.2}
\begin{theorem}
Let $(M, g)$ be a closed manifold with umbilic boundary and the conformal invariants $k_{(P_4, P_3)}$ and $Y(M, \partial M, [g])$ positive. Then there exists conformal deformations $w_1$, $w_2$ such that

$\begin{cases} 
Q_{w_1} \equiv \frac{k_{(P_4, P_3)}}{Vol(M, g_{w_1})}, \, R_{w_1}>0  \text{ on $M$} \\ 
T_{w_1} \equiv 0, \, H_{w_1} \equiv 0  \text{ on $\partial M$}
\end{cases}$
 and    \,
$\begin{cases} 
Q_{w_2} \equiv 0, \, R_{w_2}>0 \text{ on $M$} \\ 
T_{w_2} \equiv \frac{k_{(P_4, P_3)}}{Vol(\partial M, g_{w_2})}, \, H_{w_2} \equiv 0 \text{ on $\partial M$}.
\end{cases}$

\end{theorem}

\begin{proof}
To apply the degree theory, we use Theorem \ref{thm:b1compactness} and Theorem \ref{thm:b2compactness}. However, unlike the proof of Theorem \ref{thm:compactthm}, we need to use the Hopf's strong maximum principle to demonstrate that the degree is well-defined and invariant with respect to the parameter $t$.
\end{proof}


\begin{thebibliography}{100}
\bibitem[GHL]{GHL} M. J. Gursky, F. B. Hang and Y. J. Lin. Riemannian manifolds with positive Yamabe invariant and Paneitz operator. International Mathematics Research Notices 2015;
doi:10.1093/imrn/rnv176.

\bibitem[BO]{BO} Branson and B. Orsted, Explicit functional determinants in four dimensions, Proc. A.M.S. 113(1991), 669-682.

\bibitem[CLN]{CLN}  B. Chow, P. Lu, L. Ni, Hamilton’s Ricci flow, Lectures in Contemporary Mathematics, Science Press, Beijing

\bibitem[CN]{CN}Catino, G., Ndiaye, C.B.: Integral pinching results for manifolds with boundary. Ann. Sci. Norm. Super
Pisa Cl. Sci. 9(4), 785–813 (2010)

\bibitem[CGY]{CGY} S. Y. A. CHANG, M. J. GURSKY, and P. YANG, An equation of Monge–Ampere type in conformal geometry,
and four–manifolds of positive Ricci curvature. Annals of Math., 155 (2002), 711–789.

\bibitem[CQ1]{CQ1}  Chang S. Y. A., Qing J., The zeta functional determinants on manifolds with boundary. I. The
formula, J. Funct. Anal. 147 (1997), no. 2, 327-362.

\bibitem[CQ2]{CQ2}  Chang S. Y. A., Qing J., The zeta functional determinants on manifolds with boundary. II. Extremal
metrics and compactness of isospectral set, J. Funct. Anal. 147 (1997), no. 2, 363-399.

\bibitem[CQY]{CQY} S.Y.A. Chang, J. Qing and P. Yang; On the Chern-Gauss-Bonnet integral for conformal metrics on R4, Duke Math. Jour. 103 (2000), pp 523-544.

\bibitem[CY]{CY} S. Y. Chang and P. C. Yang. Extremal metrics of zeta function determinants on 4-
manifolds. Ann. of Math. (2) 142 (1995), no. 1, 171–212.

\bibitem[DM]{DM} Djadli Z., Malchiodi A., Existence of conformal metrics with constant Q-curvature, Ann. of Math.
168 (2008), 813858.

\bibitem[E]{E} J. Escobar. The Yamabe problem on manifolds with boundary. J. Differential Geom., 35(1):21–84, 1992.

\bibitem[G1]{G1} Gursky, M.J.: The Weyl functional, de Rham cohomology, and Kähler-Einstein metrics. Ann. Math. 148 (1998),
315–337.

\bibitem[G2]{G2} The principal eigenvalue of a conformally invariant differential operator, with an application to semilinear elliptic PDE, Comm. Math. Phys. 207 (1999), 131–143.

\bibitem[GS]{GS} M.J. Gursky and J. Streets. A formal riemannian structure on conformal classes and
uniqueness for the sigma2–yamabe problem. Geometry and Topology, 22(6):3501–3573, 2018

\bibitem[HY]{HY} F. B. Hang and P. C. Yang. Lectures on the fourth-order Q curvature equation. Lect.
Notes Ser. Inst. Math. Sci. Natl. Univ. Singap., 31 (2016), 1–33.


\bibitem[L]{L} C. S. Lin, A classification of solutions of a conformally invariant fourth order equation in Rn. Comment.
Math. Helv. 73, (1998) 206–231 

\bibitem[M]{M}  A. Malchiodi; Compactness of solutions to some geometric fourth-order equations, J. Reine Angew. Math., 594 (2006), 137-174.

\bibitem[N1]{N1} Ndiaye C.B., Conformal metrics with constant Q-curvature for manifolds with boundary, Comm.
Anal. Geom. 16 (2008), no. 5, 1049–1124.

\bibitem[N2]{N2} Ndiaye C.B., Constant T -curvature conformal metric on 4-manifolds with boundary, Pacific J. Math.
240 (2009), no. 1, 151–184.


\bibitem[UV]{UV} K. K. Uhlenbeck and J. A. Viaclovsky, Regularity of weak solutions to critical
exponent variational equations, Math. Res. Lett. 7 (2000), 651–656.


\bibitem[X]{X} Xu X. Uniqueness and non-existence theorems for conformally invariant equations. J Funct
Anal, 2005, 222(1): 1–28.








\end{thebibliography}
\end{document}